\documentclass[a4paper]{amsart}

\usepackage{amsmath,amssymb,latexsym,verbatim,stmaryrd,xcolor,microtype}
\usepackage[all]{xy}%\CompileMatrices\CompilePrefix{xymatrix/diagram}
\usepackage[linktoc=page,colorlinks,linkcolor={red!80!black},citecolor={red!80!black},urlcolor={blue!80!black}]{hyperref}
\usepackage{etoolbox}
\makeatletter\patchcmd{\@startsection}{\@afterindenttrue}{\@afterindentfalse}{}{}\makeatother    %omit indentation of the first paragraph of a section

\usepackage{mathptmx}
\DeclareSymbolFont{cmletters}{OML}{cmm}{m}{it}              
\DeclareSymbolFont{cmsymbols}{OMS}{cmsy}{m}{n}
\DeclareSymbolFont{cmlargesymbols}{OMX}{cmex}{m}{n}
\DeclareMathSymbol{\myjmath}{\mathord}{cmletters}{"7C}     \let\jmath\myjmath %Defining the missing commands: \jmath, \amalg and \coprod
\DeclareMathSymbol{\myamalg}{\mathbin}{cmsymbols}{"71}     
\DeclareMathSymbol{\mycoprod}{\mathop}{cmlargesymbols}{"60}\let\coprod\mycoprod
\DeclareMathSymbol{\myalpha}{\mathord}{cmletters}{"0B}     \let\alpha\myalpha %Greek letters from Computer Modern
\DeclareMathSymbol{\mybeta}{\mathord}{cmletters}{"0C}      \let\beta\mybeta
\DeclareMathSymbol{\mygamma}{\mathord}{cmletters}{"0D}     \let\gamma\mygamma
\DeclareMathSymbol{\mydelta}{\mathord}{cmletters}{"0E}     \let\delta\mydelta
\DeclareMathSymbol{\myepsilon}{\mathord}{cmletters}{"0F}   \let\epsilon\myepsilon
\DeclareMathSymbol{\myzeta}{\mathord}{cmletters}{"10}      \let\zeta\myzeta
\DeclareMathSymbol{\myeta}{\mathord}{cmletters}{"11}       \let\eta\myeta
\DeclareMathSymbol{\mytheta}{\mathord}{cmletters}{"12}     \let\theta\mytheta
\DeclareMathSymbol{\myiota}{\mathord}{cmletters}{"13}      \let\iota\myiota
\DeclareMathSymbol{\mykappa}{\mathord}{cmletters}{"14}     \let\kappa\mykappa
\DeclareMathSymbol{\mylambda}{\mathord}{cmletters}{"15}    \let\lambda\mylambda
\DeclareMathSymbol{\mymu}{\mathord}{cmletters}{"16}        \let\mu\mymu
\DeclareMathSymbol{\mynu}{\mathord}{cmletters}{"17}        \let\nu\mynu
\DeclareMathSymbol{\myxi}{\mathord}{cmletters}{"18}        \let\xi\myxi
\DeclareMathSymbol{\mypi}{\mathord}{cmletters}{"19}        \let\pi\mypi
\DeclareMathSymbol{\myrho}{\mathord}{cmletters}{"1A}       \let\rho\myrho
\DeclareMathSymbol{\mysigma}{\mathord}{cmletters}{"1B}     \let\sigma\mysigma
\DeclareMathSymbol{\mytau}{\mathord}{cmletters}{"1C}       \let\tau\mytau
\DeclareMathSymbol{\myupsilon}{\mathord}{cmletters}{"1D}   \let\upsilon\myupsilon
\DeclareMathSymbol{\myphi}{\mathord}{cmletters}{"1E}       \let\phi\myphi
\DeclareMathSymbol{\mychi}{\mathord}{cmletters}{"1F}       \let\chi\mychi
\DeclareMathSymbol{\mypsi}{\mathord}{cmletters}{"20}       \let\psi\mypsi
\DeclareMathSymbol{\myomega}{\mathord}{cmletters}{"21}     \let\omega\myomega
\DeclareMathSymbol{\myvarepsilon}{\mathord}{cmletters}{"22}\let\varepsilon\myvarepsilon
\DeclareMathSymbol{\myvartheta}{\mathord}{cmletters}{"23}  \let\vartheta\myvartheta
\DeclareMathSymbol{\myvarpi}{\mathord}{cmletters}{"24}     \let\varpi\myvarpi
\DeclareMathSymbol{\myvarrho}{\mathord}{cmletters}{"25}    \let\varrho\myvarrho
\DeclareMathSymbol{\myvarsigma}{\mathord}{cmletters}{"26}  \let\varsigma\myvarsigma
\DeclareMathSymbol{\myvarphi}{\mathord}{cmletters}{"27}    \let\varphi\myvarphi

\theoremstyle{plain}
\newtheorem{thm}{Theorem}[section]
\newtheorem{cor}[thm]{Corollary}
\newtheorem{lemma}[thm]{Lemma}
\newtheorem{prop}[thm]{Proposition}
\newtheorem*{thm*}{Theorem}
\newtheorem*{thmA}{Theorem A}

\theoremstyle{definition}
\newtheorem{df}[thm]{Definition}
\newtheorem{hypothesis}[thm]{Hypothesis}

\newtheorem{rem}[thm]{Remark}
\newtheorem{ex}[thm]{Example}

\DeclareMathOperator{\spec}{spec}
\DeclareMathOperator{\Spec}{Spec}

\DeclareMathOperator{\Sch}{Sch}

\DeclareMathOperator{\Hom}{Hom}

\DeclareMathOperator{\Sym}{Sym}

\DeclareMathOperator{\colim}{colim}

\DeclareMathOperator{\eq}{eq}

\DeclareMathOperator{\Comm}{{Comm}}
\DeclareMathOperator{\Aff}{{Aff}}
\DeclareMathOperator{\Pre}{{Pr}}
\DeclareMathOperator{\Sh}{{Sh}}

\DeclareMathOperator{\Atlas}{{Atlas}}
\DeclareMathOperator{\bp}{{Blpr}}
\DeclareMathOperator{\bpspaces}{{BlprSp}}
\DeclareMathOperator{\Rings}{{Rings}}
\DeclareMathOperator{\SRings}{{SRings}}
\DeclareMathOperator{\Sets}{{Sets}}

\newcommand\Mod{\mathcal{M}od}

\newcommand\C{{\mathbb C}}
\newcommand\F{{\mathbb F}}

\newcommand\N{{\mathbb N}}

\newcommand\R{{\mathbb R}}

\newcommand\Z{{\mathbb Z}}
\newcommand\cA{{\mathcal A}}
\newcommand\cB{{\mathcal B}}
\newcommand\cC{{\mathcal C}}
\newcommand\cD{{\mathcal D}}

\newcommand\cF{{\mathcal F}}
\newcommand\cG{{\mathcal G}}

\newcommand\cO{{\mathcal O}}

\newcommand\cP{{\mathcal P}}

\newcommand\cR{{\mathcal R}}
\newcommand\cS{{\mathcal S}}

\newcommand\cU{{\mathcal U}}
\newcommand\cV{{\mathcal V}}
\newcommand\cW{{\mathcal W}}

\newcommand\fb{{\mathfrak b}}

\newcommand\fm{{\mathfrak m}}
\newcommand\fp{{\mathfrak p}}

\newcommand\Fun{{\F_1}}

\newcommand\id{\textup{id}}

\newcommand\blanc{-}

\newcommand\can{{\textup{can}}}
\newcommand\geo{{\textup{geo}}}

\renewcommand\={\equiv}

\newcommand\alg{\textup{alg}}
\newcommand\tot{\textup{tot}}

\newcommand\op{\textup{op}}

\newdir{ >}{{}*!/-5pt/@^{(}}\newcommand{\arincl}[1]{\ar@{ >->}@<-0,0ex>#1} %inclusion arrow for xy-matrix with better spacing

\newcommand{\gen}[1]{\langle #1 \rangle}

\newcommand{\bpquot}[2]{#1\!\sslash\!#2}
\newcommand{\bpgenquot}[2]{#1\!\sslash\!\gen{#2}}

\title{Blue schemes, semiring schemes,\\ and relative schemes after To\"en and Vaqui\'e}

\author{Oliver Lorscheid}
\email{\href{mailto:oliver@impa.br}{oliver@impa.br}}
\address{Instituto Nacional de Matem\'atica Pura e Aplicada, Estrada Dona Castorina 110, Rio de Janeiro, Brazil}

\begin{document}

%\ \vspace{-2cm}

\begin{abstract}
 It is a classical insight that the Yoneda embedding defines an equivalence of schemes as locally ringed spaces with schemes as sheaves on the big Zariski site. Similarly, the Yoneda embedding identifies monoid schemes (or $\mathbb{F}_1$-schemes in the sense of Deitmar) with schemes relative to sets (in the sense of To\"en and Vaqui\'e). 
 
 In this paper, we investigate the generalization to blue schemes and to semiring schemes. We establish Yoneda functors for both schemes theories. These functors fail, however, to be equivalences in both situations. The reason for this failure is a divergence in the Grothendieck pretopologies coming from schemes as topological spaces and schemes as sheaves.
 
 Restricted to blue schemes that are locally of finite type over a blue field, we construct an inverse to the Yoneda functor, which establishes an equivalence for this subclass of blue schemes. Moreover, we verify the compatibility of the Yoneda functors with the base extension from blue schemes to semiring schemes and with the base extension from semiring schemes to usual schemes.
\end{abstract}

\maketitle
\tableofcontents

%%%%%%%%%%%%%%%%%%%%%%%%%%%%%%%%%%%%%%%%%%%%%%%%%%%%%%%%%%%%%%%%%%%%%%%%%%%%%%%%%%%%%%%%%%%%%%%%%%%%%%%%%%%%%%%%%%%%%%%%%%%%%%%%%%%%%%%%%%%%%%%%%%

\section*{Introduction}

\subsection*{Motivation}
Usually a scheme is defined as a locally ringed space that is locally isomorphic to the spectra of rings. Alternatively, the Yoneda embedding identifies a scheme $X$ with its functor of points, which is a sheaf on the big Zariski site for rings. This makes it possible to describe schemes in a completely functorial language as locally representable sheaves
on the big Zariski site; see Demazure and Gabriel's book \cite{Demazure-Gabriel70}. 

To\"en and Vaqui\'e generalize in \cite{Toen-Vaquie09} this functorial viewpoint from rings, which are commutative monoids in $\Z$-modules, to commutative monoids in any complete and cocomplete closed symmetric monoidal category $\cC$. This yields the notion of a scheme relative to $\cC$ as a locally representable sheaf on the big Zariski site for commutative monoids in $\cC$.

This produces, in particular, the notion of an $\Fun$-scheme as a scheme relative to sets together with the Cartesian product. Vezzani shows in \cite{Vezzani12} that taking the functor of points establishes an equivalence between monoidal schemes, aka, $\Fun$-schemes in Deitmar's sense (\cite{Deitmar05}), and $\Fun$-schemes in To\"en and Vaqui\'e's sense.

The purpose of this text is to extend this relation to the realms of blueprints and of semirings where both approaches to schemes play an important role. While the geometric approach to blue schemes via blueprinted space from \cite{blueprints1} embraces the theory to schemes and realizes the expected sets of $\Fun$-rational points as part of the underlying topological spaces (cf.\ \cite{LL12}, \cite{blueprints2} and \cite{blueprintedview}), To\"en and Vaqui\'e's context is the natural framework to pursue sheaf theory.

However, the equivalence of Deitmar's and To\"en-Vaqui\'e's $\Fun$-schemes does not extend to blue schemes, but these two viewpoints lead to different theories of blue schemes. The results of this text provide methods to bridge the gap and to pass from one side to the other.

The discrepancy between these two different approaches is already clear from the outset: while the Zariski topology for the category $\bp$ of blueprints coming from To\"en and Vaqui\'e's theory is subcanonical, i.e.\ blueprints define sheaves, the approach in \cite{blueprints1} leads to a Grothendieck pretopology for $\bp$ that is not subcanonical. To distinct between these two Grothendieck pretopologies for $\bp$, we call the former the \emph{subcanonical Zariski topology} and the latter the \emph{geometric Zariski topology}. Correspondingly, we will talk about \emph{subcanonical blue schemes} if we refer to To\"en and Vaqui\'e's theory, and we call blue schemes from \cite{blueprints1} \emph{geometric blue schemes}.

\subsection*{Setup}
To give a first impression to the reader that is not yet familiar with the terminology, we describe the relevant objects in brevity. Complete definitions can be found in the indicated sections of the main text. 

A \emph{blueprint} is a pair $B=(A,\cR)$ of a multiplicative monoid $A$ and an equivalence relation $\cR$ on the semiring $\N[A]$ of finite formal sums $\sum a_i$ of elements $a_i$ in $A$ that satisfy a small set of axioms (section \ref{section: blueprints}). The equivalence relation $\cR$ lets us talk about equalities $\sum a_i\=\sum b_j$ between the sums of elements $a_i$ and $b_j$ in $A$. Blueprints together with structure preserving maps form the category $\bp$.

Note that if there is a unique $c\in A$ for every $a,b\in A$ such that $a+b\=c$, then $B$ is naturally a semiring, and every (commutative) semiring (with $0$ and $1$) is of this form. Conversely, every  blueprint $B=(A,\cR)$ is a blueprint for the semiring $B^+=\N[A]/\cR$ in the literal sense. This yields a functor $(-)^+:\bp\to\SRings$ to the category of semirings.

The spectrum $X=\Spec B$ of a blueprint is the space of all prime $k$-ideals of $B$, endowed with a topology of Zariski open subsets and a structure sheaf $\cO_X$. A \emph{geometric blue scheme} is a \emph{blueprinted space}, i.e.\ a topological space together with a sheaf in $\bp$, that is locally isomorphic to the spectra of blueprints (section \ref{section: geometric blue schemes}). We denote the category of geometric blue schemes by $\Sch_\Fun^\geo$.

A \emph{subcanonical blue scheme} is a scheme relative to the category $\Mod\Fun$ of blue $\Fun$-modules (section \ref{section: blue modules}). By definition, this is a sheaf on $\bp$ with respect to the subcanonical Zariski topology that is covered by representable sheaves (sections \ref{section: relative schemes} and \ref{section: subcanonical blue schemes}). We denote the category of subcanonical blue schemes by $\Sch_\Fun^\can$.

The base extension functor $(-)^+:\bp\to\SRings$ induces corresponding functors $(-)^+:\Sch_\Fun^\geo\to\Sch_\N^{+,\geo}$ and $(-)^+:\Sch_\Fun^\can\to\Sch_\N^{+,\can}$ to the respective categories of geometric and subcanonical semiring schemes. To recall, a geometric semiring scheme is a of geometric blue scheme whose blueprints of local sections are semirings. In other words, the category of geometric semiring schemes comes with an embedding $\iota:\Sch_\N^{+,\geo}\to\Sch_\Fun^\geo$ as a full subcategory, and $(-)^+$ is left inverse and right adjoint to this embedding. A subcanonical semiring is a scheme relative to the category $\Mod^+\N$ of commutative unital semigroups. A subcanonical semiring scheme does not have a natural interpretation as a subcanonical blue scheme.

Both $\Sch_\N^{+,\can}$ and $\Sch_\N^{+,\geo}$ contain the category $\Sch_\Z^+$ of usual schemes as a full subcategory and come with respective base extension functors $-\otimes_\N\Z:\Sch_\N^{+,\can}\to\Sch_\Z^+$ and $-\otimes_\N\Z:\Sch_\N^{+,\geo}\to\Sch_\Z^+$.

The relationship between these different categories of schemes can be summarized as follows.

\begin{thmA}
 There is a diagram of functors
 \[
  \xymatrix@R=2pc@C=4pc{\Sch_\Fun^\can \ar[rr]^{\cG} \ar[d]^{(-)^+} &   & \Sch_\Fun^\geo \ar[d]<0,6ex>^{(-)^+}\ar@{-->}@/^1pc/[ll]^\cF  \\
                        \Sch_\N^{+,\can} \ar[rr]^{\cG^+}\ar[dr]_{-\otimes_\N\Z} &   & \Sch_\N^{+,\geo} \ar[u]<0,6ex>^{\iota}\ar[dl]^{-\otimes_\N\Z} \\
                                                             & \Sch_\Z^+\ar[ul]<-1ex>_{\iota}\ar[ur]<1ex>^{\iota}  }
 \]
 satisfying the following properties: 
 \begin{enumerate}
  \item\label{A1} the outer square and both the inner and the outer triangle commute; 
  \item\label{A2} the embeddings $\iota$ are left inverse and right adjoint to the respective base extension functors $(-)^+$ and $-\otimes_\N\Z$; 
  \item\label{A3} the functor $\cG$ is induced by the identity functor on $\bp$ and $\cG^+$ is induced by the identity functor on $\SRings$; both functors are essentially surjective;
  \item\label{A4} $\cF$ is a partially defined right inverse functor to $\cG$ whose domain includes all monoidal schemes and all geometric blue schemes that are locally of finite type over a blue field.
 \end{enumerate}
\end{thmA}

\subsection*{Outline of the proof} 
The major work in proving Theorem A consists in the constructions of the functors $\cG$, $\cG^+$ and $\cF$. Once constructed, the claims in \eqref{A1}--\eqref{A4} follow easily. Note that \eqref{A3} gains a precise meaning from the following explanations.

The main tool for the construction of $\cG$, $\cG^+$ and $\cF$ are \emph{affine presentations} (section \ref{section: affine presentations}), which are diagrams of open immersions between affines whose colimit equals the blue scheme or semiring scheme in question. This allows us to encounter the comparison of subcanonical blue schemes, which are objects in the category $\Sh(\bp)$ of sheaves on $\bp$, with geometric blue schemes, which are object in the category $\bpspaces$ of blueprinted spaces, in terms of diagrams of spectra of blueprints. 

This method of comparison works since subcanonical blue schemes coincide with colimits of affine presentations in $\Sh(\bp)$ (Theorem \ref{thm: relative schemes as colimits of affine presentations}) and geometric blue schemes coincide with colimits of affine presentations in $\bpspaces$ (Theorem \ref{thm: blue schemes as colimits of affine presentations}).

After unravelling the notions involved in the definition of a relative scheme for the case of blueprints (Propositions \ref{prop: blueprint morphisms od finite presentation}, \ref{prop: flat equals localization} and \ref{prop: affine subcanonical blue schemes are geometrically local}), we derive a description of subcanonical blue schemes as blueprinted spaces (section \ref{section: the blueprinted space of a subcanonical blue scheme}) and conclude that the subcanonical Zariski topology on $\bp$ is coarser than the geometric Zariski topology (Theorem \ref{thm: the geometric Zariski topology is finer than the subcanonical Zariski topology}). Applying a general fact about affine presentations (Lemma \ref{lemma: extension of a morphism of sites to a functor between schemes}) yields the functor $\cG:\Sch_\Fun^\can\to\Sch_\Fun^\geo$ (Corollary \ref{cor: functor from subcanonical to geometric blue schemes}), which sends the colimit of an affine presentation in $\Sh(\bp)$ to the colimit of the same affine presentation in $\bpspaces$. In this sense, $\cG$ is induced by the identity functor on $\bp$, as claimed in \eqref{A3}.

In order to compare the different theories of semiring schemes, we employ a result of Marty (\cite{Marty07}, cited as Theorem \ref{thm: locale of a the subcanonical spectrum of a semiring}) that gives an explicit characterization of the subcanonical Zariski topology on $\SRings$. This lets us prove that also in the case of semirings, the subcanonical Zariski topology is coarser than the geometric Zariski topology (Theorem \ref{thm: morphism of sites of affine semiring schemes}). This yields the functor $\cG^+:\Sch_\N^{+,\can}\to\Sch_\N^{+,\geo}$ (Corollary \ref{cor: functor from subcanonical semiring schemes to geometric semiring schemes}), which satisfies the claim of \eqref{A3} in the same sense as $\cG$ satisfies the corresponding claim.

Thanks to the constructions of $\cG$ and $\cG^+$ in terms of affine presentations, parts \eqref{A1} and \eqref{A2} of the theorem are readily verified (Theorem \ref{thm: compatibility with base extensions}).

The construction of the partially defined inverse $\cF$ to $\cG$ is more subtle. This functor is defined on the full subcategory of \emph{algebraically presented blue schemes}, which are, roughly speaking, geometric blue schemes that can be covered by the spectra of local blueprints such that their pairwise intersections are principal opens. Examples of algebraically presented blue schemes are monoidal schemes (Example \ref{ex: monoidal schemes are algebraically presented}) and geometrically blue schemes that are locally of finite type over a blue field (Proposition \ref{prop: blue schemes lof over a blue field are algebraically presented}). 

A certain fact on the stability under refinements (Proposition \ref{prop: common refinement for algebraic presentations}) allows the construction of the functor $\cF$ from the category of algebraically presented blue schemes to $\Sch_\Fun^{\can}$ (section \ref{section: algebraically presented and subcanonical blue schemes}). Part \eqref{A4} follows easily from the construction of $\cF$ (Theorem \ref{thm: comparison theorem}).

\subsection*{Remarks and examples} 
We include various limiting examples and remarks in the text to explain certain difficulties in the constructions of the functors $\cG$, $\cG^+$ and $\cF$.

In the final section \ref{section: concluding remarks}, we describe a different partially defined inverse $\cF'$ to $\cG$. It seems to be less important for applications than $\cF$, but we find its existence worth mentioning.

\subsection*{Acknowledgements} I would like to thank James Borger, Bertrand To\"en, Alberto Vezzani and two anonymous referees for their remarks on previous versions of this text. I would like to thank Matias del Hoyo and Andrew Macpherson for their explanations on simplicial sets and topos theory.

%%%%%%%%%%%%%%%%%%%%%%%%%%%%%%%%%%%%%%%%%%%%%%%%%%%%%%%%%%%%%%%%%%%%%%%%%%%%%%%%%%%%%%%%%%%%%%%%%%%%%%%%%%%%%%%%%%%%%%%%%%%%%%%%%%%%%%%%%%%%%%%%%%

\section{Affine presentations}
\label{section: affine presentations}

Typically a theory of scheme-like objects looks as follows: with a category $\cB$ of \emph{algebraic objects}, one associates a category $\cA$ of \emph{affine schemes} that is anti-equivalent to $\cB$. This category $\cA$ of affine schemes carries a Grothendieck pretopology and is embedded as a full subcategory in a category $\cS$, which allows us to glue affine schemes in some way. The purpose of this section is to abstract this mechanism of gluing as taking colimits over certain commutative diagrams in $\cA$, which we call affine presentations. After setting up the definitions, we formulate Hypothesis \ref{hypothesis}, which gathers certain properties \eqref{part1}--\eqref{part6} that are usually satisfied in a scheme-like theory.

Let $\cA$ be a category with finite limits, endowed with a Grothendieck pretopology. We call a morphism $W\to U$ an \emph{open map} if it occurs in a covering family of $U$.

Let $\cU$ be a commutative diagram in $\cA$. The \emph{thin category spanned by $\cU$} is the category $\langle\cU\rangle$ whose objects are the objects of $\cU$ and whose morphisms are all possible compositions of morphisms of $\cU$ in $\cA$, including the identity morphisms as compositions. Since $\cU$ is commutative, the morphism sets $\Hom(U,V)$ of $\cC(\cU)$ have at most one element, i.e.\ $\langle\cU\rangle$ is a thin category. The commutative diagram $\cU$ can be considered as a subdiagram of $\langle\cU\rangle$.

An object $U$ of $\cU$ is called a \emph{maximal object} if any morphism in $\gen\cU$ with domain $U$ is an isomorphism. We denote the set of maximal objects of $\cU$ by $\cU_{\max}$. By definition, we have $\cU_{\max}=\gen\cU_{\max}$. We say that $\cU$ is \emph{with enough maximal objects} if for every object $V$ of $\cU$, there is a morphism $V\to U$ into a maximal object $U$ in $\gen\cU$.

Let $U_0$ and $U_1$ be two objects of $\cU$. A \emph{path in $\cU$ from $U_0$ to $U_1$} is a sequence
\[
 \xymatrix@C=4pc{U_0 \ar@{-}[r]^{\varphi_1} & V_1 \ar@{-}[r]^{\varphi_2} & \dotsb \ar@{-}[r]^{\varphi_{n-1}} & V_{n-1} \ar@{-}[r]^{\varphi_n} & U_1}
\]
of objects and morphisms in $\cU$ whose arrows are allowed to have any orientation. Note that $\cV$ is itself a diagram and comes with a map $\cV\to\cU$ of diagrams, which does not have to be injective. The limit $\lim\cV$ of $\cV$ in $\cA$ comes with canonical projections $\lim\cV\to U_0$ and $\lim\cV \to U_1$, which we call the \emph{beginning} and the \emph{end of $\cV$}, respectively.

A \emph{monodromy-free diagram in $\cA$} is a diagram $\cU$ in $\cA$ such that for any object $U$ in $\cU$ and any path $\cV$ in $\cU$ from $U$ to $U$, the canonical morphism
\[
 \lim\bigl( \lim\cV \!\! \begin{array}{c} \longrightarrow\vspace{-8pt} \\ \longrightarrow \end{array} U \bigr) \quad \longrightarrow \quad \lim\cV
% \xymatrix{ \lim \bigl(\xymatrix{\cV \ar@<0.5ex>[r] \ar@<0.5ex>[r] & U} \bigr) \ar[r] & \lim \cV}
\]
is an isomorphism where the two arrows $\lim\cV\to U$ are the beginning and the end of $\cV$. Note that a monodromy-free diagram is commutative.

An \emph{affine presentation in $\cA$} is a monodromy-free diagram $\cU$ of open morphisms in $\cA$ with enough maximal objects. A \emph{morphism $\cU\to \cV$ of affine presentations $\cU$ and $\cV$} is a family of morphisms $\{\varphi_U:U\to V(U)\}$ from all objects $U$ in $\cU$ to some objects $V(U)$ in $\cV$ such there is a morphism $V(U_1)\to V(U_2)$ in $\cV$ for every morphism $U_1\to U_2$ in $\cU$ such that the resulting square
 \[
  \xymatrix@R=1pc@C=4pc{U_1 \ar[d]\ar[r] & V(U_1)\ar[d] \\ U_2 \ar[r] & V(U_2)}
 \]
commutes.

Let $\cV$ be an affine representation. A \emph{refinement of $\cV$} is a morphism $\Phi:\cU\to \cV$ of affine presentations such that for all $V$ in $\cV$, the family $\{\Phi:U\to \Phi(U) \}$ with $\Phi(U)=V$ is a covering family of $V$. Note that the natural morphism $\cU\to\langle\cU\rangle$ is an refinement.

An affine presentation $\cU$ is called an \emph{atlas} if it has no composable morphisms and every object $W$ occurs as the domain of at most two morphisms $W\to U$ and $W\to V$. Informally speaking, an atlas is a collection $\cU_{\max}$ of objects together with a covering for every intersection of two maximal objects $U$ and $V$. It is not hard to see that every affine presentation $\cV$ admits a refinement $\cU\to \cV$ by an atlas $\cU$.

\begin{lemma}\label{lemma: fibre products of affine presentations}
 The category of affine presentations in $\cA$ contains fibre products and the base change of a refinement is a refinement.
\end{lemma}

\begin{proof}
 Let $\Phi:\cU\to\cV$ and $\Psi:\cV'\to\cV$ be two morphisms of affine presentations. We construct the fibre product $\cU'=\cU\times_\cV\cV'$ as follows. The objects of $\cU'$ are the fibre products $U\times_VV'$ for every pair of morphisms $U\to V$ in $\Phi$ and $V'\to V$ in $\Psi$. The morphisms of $\cU'$ are the morphisms $U\times_VV'\to U'\times_VV'$ that are induced by morphisms $U\to U'$ in $\cU$ and the morphisms $U\times_VV'\to U\times_VV''$ that are induced by morphisms $V'\to V''$ in $\cV$. 
 
 It is not hard to verify that $\cU'$ is an affine presentation, which satisfies the universal property of the fibre product $\cU\times_\cV\cV'$ with respect to the canonical morphisms $\Phi':\cU'\to\cV'$ and $\Psi':\cU'\to \cU$.

 That refinements are stable under base changes follows from the axiom of a Grothen\-dieck pretopology that ensures that covering families are stable under base change.
\end{proof}

In theories of scheme-like objects, like To\"en and Vaqui\'e's theory of relative schemes or the author's theory of blue schemes, one faces typically the following situation. There is a fully faithful embedding $\cA\to\cS$ of the category of affine schemes into some larger category $\cS$, which often is a category of sheaves on $\cA$ or a category of topological spaces together with a structure sheaf. The category $\Sch_\cA$ of schemes is defined in some way as a full subcategory of $\cS$ that contains $\cA$. Affine presentations work well as a tool to compare different scheme theories under the following assumptions.

\begin{hypothesis}\label{hypothesis}
 The embeddings $\cA\to\Sch_\cA\to\cS$ satisfy the following properties.
 \begin{enumerate}
  \item\label{part1} The categories $\cA$, $\Sch_\cA$ and $\cS$ contain finite limits and the functors $\cA\to \Sch_\cA$ and $\Sch_\cA\to \cS$ commute with finite limits.
  \item\label{part2} The category $\cS$ contains a colimit of every affine presentation in $\cA$, and $\Sch_\cA$ is the full subcategory of $\cS$ whose objects are colimits of affine presentations in $\cA$.
  \item\label{part3} For every scheme $X$, the functor $\Hom(\blanc, X)$ is a sheaf on $\cA$.
  \item\label{part4} Refinements $\cU\to\cV$ induce isomorphisms $\colim\cU\to\colim\cV$ of schemes. 
  \item\label{part5} Affine presentations $\cU$ and $\cV$ whose colimits are isomorphic have a common refinement $\cW\to\cU$ and $\cW\to\cV$.
  \item\label{part6} Every morphism of schemes is induced by a morphism of affine presentations.
 \end{enumerate}
\end{hypothesis}

Note that these properties are not independent. For instance, \eqref{part5} follows from the other properties. If a scheme $X$ in $\cS$ is the colimit of an affine presentation $\cU$, then we say that $\cU$ is \emph{an affine presentation of $X$}.

Let $\cA$ and $\cA'$ be categories with finite limits, equipped with Grothendieck pretopologies. Let $\cA\to\Sch_\cA\to\cS$ and $\cA'\to\Sch_{\cA'}\to\cS'$ be fully faithful embeddings that satisfy Hypothesis \ref{hypothesis}.

\begin{lemma}\label{lemma: extension of a morphism of sites to a functor between schemes} 
 Let $\cG:\cA\to\cA'$ be a functor that commutes with fibre products and that preserves covering families. Then there exists a unique functor $\overline \cG:\Sch_\cA\to \Sch_{\cA'}$ such that for all morphisms $\Phi:\cU\to\cV$ of affine presentations in $\cA$, we have a natural identification $\overline\cG(\colim\Phi)=\colim\cG(\Phi)$. In particular, this yields $\overline\cG(\colim\cU)=\colim\cG(\cU)$.
\end{lemma}

\begin{proof}
 The uniqueness follows from properties \eqref{part2} and \eqref{part6}. We have to verify that the formulas $\overline\cG(\colim\cU)=\colim\cG(\cU)$ and $\overline\cG(\colim\Phi)=\colim\cG(\Phi)$ give rise to a well-defined functor.

 Let $\cU$ an affine presentation in $\cA$. Then $\cG(\cU)$ is a commutative diagram of open immersions in $\cA'$ with enough maximal objects. Since $\cU$ is monodromy-free and $\cG$ commutes with fibre products, $\cG(\cU)$ is also monodromy-free and therefore an affine presentation. It is also clear that $\cG$ maps morphisms of affine presentations to morphisms of affine presentations.

 The definition of $\overline\cG(X)$ is independent (up to canonical isomorphism) from the choice of affine presentation $\cU$ of $X$ for the following reason. Since two affine presentation have a common refinement (property \eqref{part5}), it is enough to show that a refinement $\cW\to \cU$ of an affine presentation $\cU$ of $X$ induces an isomorphism $\colim\cG(\cW)\to\colim\cG(\cU)$.
 
 Since $\cG$ is a morphism of sites, it maps open morphisms to open morphisms and covering families to covering families. Therefore $\cG(\cW)\to \cG(\cU)$ is a refinement, and by property \eqref{part4}, $\colim\cG(\cW)\to\colim\cG(\cU)$ is an isomorphism.

 Let $\Phi':\cU'\to\cV'$ and $\Phi'':\cU''\to\cV''$ be two morphisms of affine presentations with the same colimit $\varphi:X\to Y$. Using Lemma \ref{lemma: fibre products of affine presentations}, we find a common refinement $\cU$ of $\cU'$ and $\cU''$, a common refinement $\cV$ of $\cV'$ and $\cV''$, and a morphism $\Phi:\cU\to\cV$ such that $\colim\Phi'=\colim\Phi=\colim\Phi''$. This shows that $\overline\cG(\varphi)=\colim\cG(\Phi')=\colim\cG(\Phi'')$ does not depend on the choice of affine presentation of $\varphi$.
 
 The independence from the affine presentation also shows that $\overline\cG$ commutes with the composition of morphisms.
\end{proof}

Note that $\cG$ is not required to preserve a terminal object. If, however, $\cG$ has a left adjoint, then it commutes with all limits, as it will be the case for all applications in this text.

\begin{ex}
 To illustrate the concepts of this section, we consider a fan $\Delta$ of polyhedral, strictly convex and rational cones $\tau\subset\R^n$. Let $\tau^\vee=\{x\in \R^n|\gen{x,y}\geq 0\text{ for all }y\in\tau\}$ be the dual cone and $S_\tau=\tau^\vee\cap \Z^n$ be the intersection with the dual lattice, written as a multiplicative semigroup. Let $U_\tau=\Spec\C[S_\tau]$ be the corresponding affine toric variety.
 
 An inclusion of cones $\sigma\subset \tau$ induces an open immersion $U_\sigma\to U_\tau$, which defines a diagram $\cU$ of affine schemes and open immersions. The toric variety associated with $\Delta$ is defined as the colimit $X(\Delta)=\colim \cU$ in the category of locally ringed spaces. The diagram $\cU$ is obviously commutative and with enough maximal objects. Since $\cU$ contains finite limits of all non-empty subdiagrams, it is immediate that $\cU$ is monodromy-free, using Remark \ref{rem: monodromy-free}. Thus $\cU$ is an affine presentation.
 
 The following diagram $\cV$ is an atlas whose colimit in schemes is $X(\Delta)$. Its objects are all $U_\tau$ where $\tau$ is either a maximal cone in $\Delta$ or a facet of a maximal cone. The morphisms of $\cV$ are of the form $U_\sigma\to U_\tau$ whenever $\sigma$ is a facet of a maximal cone $\tau$. Note that the natural inclusion $\cV\to\cU$ is not a refinement---however, $\cV$ and $\cU$ admit a common refinement that commutes with this natural inclusion.
\end{ex}

\begin{rem}\label{rem: monodromy-free}
 We conclude with some remarks on monodromy-free diagrams and affine presentations.

 A diagram $\cU$ is monodromy-free if and only if for any two objects $U_0$ and $U_1$ of $\cU$, any two paths $\cV$ and $\cV'$ from $U_0$ to $U_1$ and any two morphisms $\psi:W\to\lim\cV$ and $\psi':W\to \lim\cV'$, the lower square of 
 \[
  \xymatrix@R=0.5pc@C=5pc{   &  \lim\cV \ar[r]\ar[rdd]|\hole & U_0 \\  W \ar[ur]^\psi \ar[dr]_{\psi'} \\ & \lim\cV'\ar[uur]\ar[r] & U_1}
 \]
 commutes if the upper square commutes, where the morphisms to $U_0$ and $U_1$ are the beginnings and ends of $\cV$ and $\cV'$. This equivalence can be easily proven by observing that the limit of $\lim \cV\times_{U_0}\lim\cV'$ equals the limit of the concatenation $\cV''$ of $\cV$ and $\cV'$ at $U_0$.

 Verdier introduces in Expose V, section 7.3, of \cite{SGA4-2} the notion of a hypercover to talk about schemes in terms of diagrams of affine schemes. This idea is closely related to the notion of affine presentation, and can be made precise for a subcanonical pretopology on $\cA$, which is embedded into the category $\Sh(\cA)$ of sheaves on $\cA$. 
 
 Namely, an $1$-coskeletal simplicial set $\cU$ of open morphisms in $\cA$ is an affine presentation if and only if $\cU$ is an $1$-hypercover of its colimit $\colim\cU$ in $\Sh(\cA)$. This fact can be seen along the proof of Proposition \ref{thm: relative schemes as colimits of affine presentations}, and it is sufficient to work with affine presentations of this from. We refrain from this terminology, however, since we want to suppress the relation of $\cU$ to its colimit in the sheaf category.
\end{rem}

%%%%%%%%%%%%%%%%%%%%%%%%%%%%%%%%%%%%%%%%%%%%%%%%%%%%%%%%%%%%%%%%%%%%%%%%%%%%%%%%%%%%%%%%%%%%%%%%%%%%%%%%%%%%%%%%%%%%%%%%%%%%%%%%%%%%%%%%%%%%%%%%%%

\section{Relative schemes after To\"en and Vaqui\'e}
\label{section: relative schemes}

We recall the definition of relative schemes from To\"en and Vaqui\'e's paper \cite{Toen-Vaquie09}. Let $\cC$ be a closed symmetric monoidal category that is complete and cocomplete. We denote by $\Comm(\cC)$ the category of commutative, associative and unital semigroups in $\cC$. We call an object $B$ of $\Comm(\cC)$ for short a \emph{commutative monoid} if the context is clear. A \emph{$B$-algebra} is a morphism $B\to C$ of commutative monoids and a \emph{$B$-algebra morphism} is a morphism $C\to C'$ that commutes with the morphisms $B\to C$ and $B\to C'$.

A $B$-module is an object $M$ of $\cC$ together with a morphism $B\times M\to M$ in $\cC$ that satisfies usual axioms ``$(ab.m)=a.(b.m)$'' and ``$1.m=m$'' of a monoid action. A morphism $M\to N$ of $B$-modules is a morphism in $\cC$ that commutes with the actions of $B$ on $M$ and $N$, respectively. This defines the (complete and cocomplete) category $\Mod B$ of $B$-modules. 

Let $f:B\to C$ be a morphism of commutative monoids in $\cC$. The morphism $f$ is \emph{flat} if $\blanc\otimes_B C:\Mod B\to \Mod C$ commutes with finite limits and colimits. The morphism $f:B\to C$ is \emph{of finite presentation} if for all directed systems $\cD$, the canonical map
\[
 \Psi_\cD: \quad \colim \ \Hom_B (C,\cD) \quad \longrightarrow \quad \Hom_B(C,\colim \cD)
\]
is bijective. 

An \emph{affine scheme relative to $\cC$} is an object of the dual category of $\Comm(\cC)$, which we denote by $\Aff_\cC$. Let $\spec:\Comm(\cC)\to \Aff_\cC$ be the anti-equivalence of dual categories. Let $f:B\to C$ be a morphism of commutative monoids in $\cC$. Then $f^\ast:\spec C\to \spec B$ is called a \emph{Zariski open immersion} if $f:B\to C$ is a flat epimorphism of finite presentation.

A family $\{\varphi_i: \spec B_i \to\spec B\}_{i\in I}$ is a \emph{covering family} if all the morphisms $\varphi_i$ are Zariski open immersions and if there is a finite subset $J\subset I$ such that the functor
\[
 \Phi \ = \ \prod_{j\in J} \ \blanc\otimes_BB_j \ : \quad \Mod B \quad \longrightarrow \quad \prod_{j\in J} \ \Mod B_j
\]
is conservative (i.e.\ $f:M\to N$ is an isomorphism if $\Phi(f)$ is an isomorphism). This endows the category $\Aff_\cC$ of affine schemes with a Grothendieck pretopology, called the \emph{Zariski topology of $\Aff_\cC$}. This defines the full subcategory $\Sh(\Aff_\cC)$ of sheaves in the category $\Pre(\Aff_\cC)$ of pre-sheaves. 

We use the characterization of relative schemes as a quotient of a disjoint union of affine schemes by a suitable equivalence relation in order to bypass some notions that are needed in the original definition of a scheme relative to $\cC$ from \cite{Toen-Vaquie09}. Namely, we define a scheme relative to $\cC$ in terms of the following theorem.

\begin{thm}\label{thm: relative schemes as colimits of affine presentations}
 A sheaf $F$ on $\Aff_\cC$ is a \emph{scheme relative to $\cC$} if and only if it is the colimit (in $\Sh(\Aff_\cC)$) of an affine presentation in $\Aff_\cC$.
\end{thm}

\begin{proof}
 Without recalling all definitions from \cite{Toen-Vaquie09}, we outline how the proposition follows from \cite[Prop.\ 2.18]{Toen-Vaquie09}.  Assume that $F$ is a scheme relative to $\cC$. By \cite[Prop.\ 2.18]{Toen-Vaquie09}, $F$ is the quotient of a disjoint union $X=\coprod U_i$ of representable sheaves $U_i$ by an equivalence relation $R$ that satisfies the following two properties.
 \begin{enumerate}
  \item\label{equiv1} If $R_{i,j}$ is the fibre product of $R$ with $U_i\times U_j$ over $X\times X$, then the induced morphism $R_{i,j}\to U_i\times U_j\to U_i$ is an open morphism for every $i$ and $j$.
  \item\label{equiv2} The map $R_{i,i}\to U_i\times U_i$ is isomorphic to the diagonal embedding $U_i\to U_i\times U_i$. 
 \end{enumerate}
 
 We define an affine presentation $\cU$ as follows. Its maximal objects are the representable sheaves $U_i$, which we identify with their corresponding objects in $\Aff_\cC$. Then we choose a covering family $\{W_{i,j,k}\to R_{i,j}\}$ for each distinct pair of indices $i$ and $j$ and include the morphisms $W_{i,j,k}\to U_i$ and $W_{i,j,k}\to U_j$ in $\cU$. It is clear from the definition that $\cU$ is a commutative diagram of open immersions with enough maximal objects whose colimit in $\Aff_\cC$ is $F$.
 
 We are left with verifying that $\cU$ is monodromy-free. We do so by an induction over the length of paths $\cV$ in $\cU$ with equal beginning and end $U$. Since all arrows in $\cU$ map from a non-maximal object to a maximal object, the length of a path in $\cU$ is an even number $2l$ with $l\in\N$.
 
 If $2l=0$, then $\cV$ is the trivial path $U$, and the canonical morphism $\lim(\cV\!\!\!\begin{array}{c}\to\vspace{-8.5pt}\\ \to\end{array} U \bigr) \to\lim\cV$ is evidently an isomorphism.
 
 If $2l>0$ and $U$ is maximal, then $\cV$ is of the form
 \[
  \xymatrix{U & W_1 \ar[l] \ar[r] & U_1 & W_2 \ar[l] \ar[r] & U_2 & \dotsb \ar[l] \ar[r] & U'}
 \]
 Since limits of sheaves coincide with the limits as presheaves, it suffices to show that the canonical map $\lim(\cV(B)\!\!\!\begin{array}{c}\to\vspace{-8.5pt}\\ \to\end{array} U(B) \bigr) \to\lim\cV(B)$ is a bijection for every $B$ in $\Comm(\cC)$. This is the case if and only if for every sequence
 \[
  \xymatrix{x & w_1 \ar@{|->}[l] \ar@{|->}[r] & y_1 & w_2 \ar@{|->}[l] \ar@{|->}[r] & y_2 & \dotsb \ar@{|->}[l] \ar@{|->}[r] & x'}
 \]
 of elements $x,x'\in U(B)$, $y_i\in U_i(B)$ and $w_i\in W_i(B)$, we have that $x=x'$. By the transitivity of the equivalence relation $R$, there exist open immersions $W_1'\to U$ and $W_1'\to U_2$ in $\cU$ and an element $w_1'\in W_1'(B)$ that maps to $x$ and $y_2$, respectively. Thus we obtain a path
 \[
  \xymatrix{U & W_1' \ar[l] \ar[r] & U_2 & \dotsb \ar[l] \ar[r] & U'}
 \]
 of length $2l-2$ and a sequence of elements
 \[
  \xymatrix{x & w'_1 \ar@{|->}[l] \ar@{|->}[r] & y_2 & \dotsb \ar@{|->}[l] \ar@{|->}[r] & x'}.
 \]
 By the induction hypothesis, we have $x=x'$. 
 
 The case that $U$ is not a maximal element is treated analogously, but with all arrows inverted. This completes the proof that $\cU$ is an affine presentation whose colimit is $F$.
 
% By replacing $\cU$ by $\overline\cU$, as constructed above, we can assume that $\cU$ is closed under finite limits without changing $\cU_{\max}$ nor the validity of $F=\colim\cU$. As explained in Remark \ref{rem: monodromy-free}, it suffices to consider test diagrams of the form
% \[
%  \xymatrix@R=0.5pc@C=5pc{   &  V \ar[r]^(0.45){\varphi_0}\ar[rdd]^(0.4){\varphi_1}|\hole & U_0 \\  W \ar[ur]^\psi \ar[dr]_{\psi'} \\ & V'\ar[uur]_(0.4){\varphi_0'}\ar[r]_(0.45){\varphi_1'} & U_1}
% \]
% where $\varphi_0,\varphi_1,\varphi_0',\varphi_1'$ are in $\cU$ and $\psi$ and $\psi'$ are in $\Aff_\cC$. We can consider all objects in this diagrams as presheaves on $\Comm(\cC)$ and verify the commutativity of the upper and lower squares by taking values in objects $B$ of $\Comm(\cC)$.
 
% Assume that the upper diagram commutes. Then we have for every $B$ and every $w\in W(B)$ that $y=\varphi_{0,B}\circ \psi_B(w)=\varphi'_{0,B}\circ\psi'_B(w)$. Let $x=\varphi_{1,B}\circ\psi_B(w)$ and $z=\varphi'_{1,B}\circ\psi'_B(w)$. By the definition of $\cU$, this means that $x\sim y$ and $y\sim z$ with respect to the equivalence relation $R(B)$ on $X(B)$. By transitivity, we have $x\sim z$, which means that $\varphi_{1,B}\circ\psi_B(w)=\varphi'_{1,B}\circ\psi'_B(w)$. This shows that the lower square of our test diagram commutes and verifies that $\cU$ is monodromy-free. Thus $\cU$ is an affine presentation whose colimit is $F$.

 Conversely, if the sheaf $F$ on $\Aff_\cC$ has an affine presentation $\cU$, then we define an equivalence relation $R$ on $X=\coprod_{U\in \cU_{\max}} U$ as follows. As a first step, we replace $\cU$ by $\gen\cU$ and embed $\gen\cU$ into the category $\Sets^{\gen\cU}$ of presheaves in $\gen\cU$. We define $\overline{\cU}$ as the closure of $\gen\cU$ under finite limits in $\Sets^{\gen\cU}$. To avoid set theoretic problems, we replace $\overline\cU$ by an equivalent, but small category. Moreover, we can assume that the set of maximal objects in $\overline\cU$ is $\cU_{\max}$. It is clear that $\overline\cU$ is an affine presentation.
 
 For $U,V\in\cU_{\max}=\overline\cU_{\max}$, we define $R_{U,V}$ as the subsheaf of $U\times V$ that is generated by the image of $\coprod W$ where the coproduct ranges over all $W$ in $\overline\cU$ with morphisms $W\to U$ and $W\to V$ in $\overline\cU$. By the definition of a Zariski open subsheaf of an affine scheme in \cite{Toen-Vaquie09}, $R_{U,V}$ is an open subsheaf of both $U$ and $V$. We define $R$ as the disjoint union $\coprod R_{U,V}$ over all $U,V\in\cU$, which is a subsheaf of $X\times X$. 
 
% Conditions (a), (b) and (c) of \cite[Prop.\ 2.18]{Toen-Vaquie09} are clear by construction. Condition (d) follows once we know 
 We begin with showing that $R$ is an equivalence relation on $X$. Reflexivity and symmetry of $R$ are obvious, and transitivity can be established as follows.
 
 Considering the objects $U$ of $\cU$ as sheaves, we can test transitivity for the sets $U(B)$ where $B$ is an object of $\Comm(\cC)$. Given three objects $U_1$, $U_2$ and $U_3$ in $\cU$ and elements $x_i\in U_i(B)$ (for $i=1,2,3$), we have $x_i\sim x_j$ if and only if there is an $w_{i,j}\in R_{i,j}(B)$ such that $x_i=\varphi_{i,j,B}(w_{i,j})$ and $x_j=\varphi_{j,i,B}(w_{i,j})$ where $R_{i,j}=R_{U_i,U_j}$ and $\varphi_{i,j,B}:R_{i,j}(B)\to U_i(B)$ is induced by the inclusion $R_{i,j}\subset U_i\times U_j$, followed by the projection onto $U_i$.
 
 We assume that $x_1\sim x_2$ and $x_2\sim x_3$, witnessed by elements $w_{1,2}\in R_{1,2}(B)$ and $w_{2,3}\in R_{2,3}(B)$, and consider the following diagram 
 \[
  \xymatrix@R=0pc@C=6pc{                                     &                                                                & U_1(B) \\ 
                                                             & R_{1,2}(B) \ar[ru]^{\varphi_{1,2,B}} \ar[rd]^{\varphi_{2,1,B}}          \\ 
                         R_{1,2}(B)\times_{U_2(B)}R_{2,3}(B) \ar[ru]^{\psi_{1,B}} \ar[rd]_{\psi_{3,B}} &                      & U_2(B) \\ 
                                                             & R_{2,3}(B) \ar[ru]^{\varphi_{2,3,B}} \ar[rd]^{\varphi_{3,2,B}}          \\ 
                                                             &                                                                & U_3(B).      }
 \]
 Using standard arguments about coverings of sheaves, we can reduce this to the situation in which all of these sheaves are representable, in which case there is an $w_{1,3}$ in $R_{1,2}(B)\times_{U_2(B)}R_{2,3}(B)$ such that $\psi_{1,B}(w_{1,3})=w_{1,2}$ and $\psi_{3,B}(w_{1,3})=w_{2,3}$. By the definition of $R_{1,3}$, we have $w_{1,3}\in R_{1,3}(B)$, which shows that $x_1\sim x_3$ as desired.
 
 This shows that $R$ is an equivalence relation on $X=\coprod U$. By construction, it satisfies property \eqref{equiv1}, as stated in the beginning of the proof. Using that $\cU$ is monodromy-free, we can show that two different points of $U_i(B)$ cannot be identified, which establishes \eqref{equiv2}. We leave the details of this last argument to the reader. We conclude that $F=X/R$ is a scheme relative to $\cC$.
\end{proof}

Let $\Sch_\cC$ be the full subcategory of $\Sh(\Aff_\cC)$ whose objects are schemes relative to $\cC$.

\begin{prop}\label{prop: properties 1 to 6 for relative schemes}
 The embeddings $\Aff_\cC\to\Sch_\cC\to\Sh(\Aff_\cC)$ satisfy Hypothesis \ref{hypothesis}.
\end{prop}

\begin{proof}
 Property \eqref{part1} of Hypothesis \ref{hypothesis} is satisfied since $\lim\Hom(\blanc,B_i)=\Hom(\blanc,\lim B_i)$ by the universal property of limits and since schemes are stable under products of sheaves, see \cite[Prop.\ 2.18]{Toen-Vaquie09}. Property \eqref{part2} is established by Proposition \ref{thm: relative schemes as colimits of affine presentations}. Property \eqref{part3} is shown in \cite[Cor.\ 2.11]{Toen-Vaquie09} for affine schemes. By the definition of a scheme, it is a sheaf on $\Aff_\cC$. 

 Let $\cU\to\cV$ be a refinement and $X=\colim\cU$ and $Y=\colim\cV$ the corresponding schemes. Since $X$ and $Y$ are sheaves on $\Aff_\cC$, a morphism $W\to Y$ is represented by an affine presentation $\cW$ of $W$ and a morphism $\cW\to\cV$. We can assume that $\cW=\Atlas(\cW)$ since the colimit $X$ of an affine presentation $\cW$ only depends on the maximal elements and a covering of their pairwise intersections $U\times_X V$, which are given by the elements of $\cW_{U,V}$. By the base change property of Grothendieck pretopologies, the refinement $\cU\to\cV$ defines a refinement $\cW'\to\cW$ for $\cW'=\cU\times_\cV\cW$ and a morphism $\cW'\to\cU$. By the local character of Grothendieck pretopologies, $\colim\cW'=W$. With this, it is easy to verify that the induced morphism $\colim\cU\to\colim\cV$ is an isomorphism. This establishes property \eqref{part4}. 

 Let $\cU$ and $\cV$ be affine presentations such that $\colim\cU\simeq X\simeq \colim \cV$. We define a common refinement $\cW$ of $\cU$ and $\cV$ as follows. By \eqref{part4}, we can assume that $\cU$ and $\cV$ are affine atlases. For every $U\in\cU$ and $V\in\cV$, the common open subscheme $W=U\times_X V$ of $U$ and $V$ can be covered by affine open subschemes $W_i$. If $U\to U'$ is a morphism in $\cU$, $V\in\cV$ and $\{W_i\}$ and $\{W_j'\}$ are the coverings of $U\times_XV$ and $U'\times_XV$, respectively, then we can refine the covering $\{W_i\}$ such that the induced morphism $U\times_XV\to U'\times_XV$ restricts to morphisms $W_i\to W_{j(i)}$ for all $i$. The same argument holds for morphisms $V\to V'$. All the $W_i$ together with the morphisms $W_i\to W_{j(i)}$ yield a diagram $\cW$, which is commutative since there are no morphisms to compare. Since open immersions are stable under base change (\cite[Lemme 2.13]{Toen-Vaquie09}) all morphisms of $\cW$ are open immersions. The maximal elements of $\cW$ are the $W_i$ that cover $U\times_XV$ for some $U\in\cU_{\max}$ and $V\in\cV_{\max}$. The monodromy condition follows easily from the monodromy condition for $U$ and $V$. Therefore $\cW$ is an affine presentation, and indeed an affine atlas, that is a common refinement of $\cU$ and $\cV$ with respect to the canonical morphisms $\cW\to\cU$ and $\cW\to \cV$. This establishes property \eqref{part5}.

 By similar arguments, we find for a given morphism $\colim\cU\to\colim\cV$ of schemes an refinement $\cU'\to\cU$ and a morphism $\cU'\to\cV$ that induces this morphism of schemes. This is property \eqref{part6}.
\end{proof}

%%%%%%%%%%%%%%%%%%%%%%%%%%%%%%%%%%%%%%%%%%%%%%%%%%%%%%%%%%%%%%%%%%%%%%%%%%%%%%%%%%%%%%%%%%%%%%%%%%%%%%%%%%%%%%%%%%%%%%%%%%%%%%%%%%%%%%%%%%%%%%%%%%

\section{Blueprints}
\label{section: blueprints}

We recall the definition of a blueprint. Note that we follow the convention of \cite{blueprints2}, i.e.\ all blueprints are \emph{proper} and \emph{with zero}, according to the terminology in \cite{blueprints1} .

By a \emph{monoid with zero}, we mean a multiplicatively written commutative semigroup $A$ with a neutral element $1$ and an absorbing element $0$, which are characterized by the properties $1\cdot a=a$ and $0\cdot a=0$ for all $a\in A$. A \emph{morphism of monoids with zero} is a multiplicative map $f:A_1\to A_2$ that maps $1$ to $1$ and $0$ to $0$.

A \emph{blueprint $B$}\index{Blueprint} is a monoid $A$ with zero together with a \emph{pre-addition}\index{Pre-addition} $\cR$, i.e.\ $\cR$ is an equivalence relation on the semiring $\N[A]=\{\sum a_i|a_i\in A\}$ of finite formal sums of elements of $A$ that satisfies the following axioms (where we write $\sum a_i\=\sum b_j$ whenever $(\sum a_i,\sum b_j)\in\cR$):
\begin{enumerate}
 \item\label{ax1} $\sum a_i\=\sum b_j$ and $\sum c_k\=\sum d_l$ imply $\sum a_i+\sum c_k\=\sum b_j+\sum d_l$ and $\sum a_ic_k\=\sum b_jd_l$,
 \item\label{ax2} $0\=(\text{empty sum})$, and 
 \item\label{ax3} if $a\= b$, then $a=b$ (as elements in $A$).
\end{enumerate}
A \emph{morphism $f:B_1\to B_2$ of blueprints}\index{Morphism of blueprints} is a multiplicative map $f:A_1\to A_2$ between the underlying monoids of $B_1$ and $B_2$, respectively, with $f(0)=0$ and $f(1)=1$ such that for every relation $\sum a_i\=\sum b_j$ in the pre-addition $\cR_1$ of $B_1$, the pre-addition $\cR_2$ of $B_2$ contains the relation $\sum f(a_i)\=\sum f(b_j)$. Let $\bp$ be the category of blueprints.

In the following, we write $B=\bpquot A\cR$ for a blueprint $B$ with underlying monoid $A$ and pre-addition $\cR$. We adopt the conventions used for rings: we identify $B$ with the underlying monoid $A$ and write $a\in B$ or $S\subset B$ when we mean $a\in A$ or $S\subset A$, respectively. Further, we think of a relation $\sum a_i\=\sum b_j$ as an equality that holds in $B$ (without the elements $\sum a_i$ and $\sum b_j$ being defined, in general). 

Given a set $S$ of relations, there is a smallest equivalence relation $\cR$ on $\N[A]$ that contains $S$ and satisfies Axioms \eqref{ax1} and \eqref{ax2}. If $\cR$ satisfies also Axiom \eqref{ax3}, then we say that $\cR$ is the pre-addition generated by $S$, and we write $\cR=\gen S$. In particular, every monoid $A$ with zero has a smallest pre-addition $\cR=\gen\emptyset$.

More generally, let $A$ be a monoid with zero and $\cR$ an equivalence relation on $\N[A]$ that satisfies Axioms \eqref{ax1} and \eqref{ax2}. We can form the quotient set $A'=A/\sim$ where $a\sim b$ whenever $a\=b$. Then $A'$ inherits the structure of a monoid by the multiplicativity of $\cR$, and the image $\cR'$ of $\cR$ in $\N[A']\times\N[A']$ is a pre-addition on $A$, satisfying Axiom \eqref{ax3} (see Lemma 1.6 in \cite{blueprints1} for more details on the construction of the proper quotient).  We say that $\bpquot A\cR$ is a \emph{representation} of the blueprint $\bpquot{A'}{\cR'}$, and we say that the representation $\bpquot A\cR$ of $B=\bpquot{A'}{\cR'}$ is \emph{proper} if $A=A'$. 

A (commutative) semiring $R$ (with additive and multiplicative unit) defines the blueprint $B=\bpquot{A}{\cR}$ where $A=R$ as multiplicative monoid and $\cR=\{\sum a_i\=\sum b_j|\sum a_i=\sum b_j\text{ in }R\}$. This construction is functorial in $R$ and provides a fully faithful embedding of semirings into blueprints. In the following, we often consider semirings as blueprints. This embedding admits a left adjoint and left inverse functor, which associates with a blueprint $B=\bpquot{A}{\cR}$ its \emph{universal semiring} $B^+=\N[A]/\cR$. Note that $B^+$ is well-defined as semiring since $\cR$ is an equivalence relation on $\N[A]$ that respects addition and multiplication. Note further that $B$ comes with an inclusion $B\to B^+$ of blueprints that is universal for all morphisms from $B$ into a semiring.

%%%%%%%%%%%%%%%%%%%%%%%%%%%%%%%%%%%%%%%%%%%%%%%%%%%%%%%%%%%%%%%%%%%%%%%%%%%%%%%%%%%%%%%%%%%%%%%%%%%%%%%%%%%%%%%%%%%%%%%%%%%%%%%%%%%%%%%%%%%%%%%%%%

\section{Blue \texorpdfstring{$B$}{B}-modules}
\label{section: blue modules}

In this section, we introduce the notion of a blue $B$-module for a blueprint $B$.

Let $M$ be a pointed set. We denote the base point of $M$ by $\ast$. A \emph{pre-addition on $M$} is an equivalence relation $\cP$ on the semigroup $\N[M]=\{\sum a_i|a_i\in M\}$ of finite formal sums in $M$ with the following properties (as usual, we write $\sum m_i\=\sum n_j$ if $\sum m_i$ stays in relation to $\sum n_j$):
\begin{enumerate}
 \item $\sum m_i\=\sum n_j$ and $\sum p_k\=\sum q_l$ imply $\sum m_i+\sum p_k\=\sum n_j+\sum q_l$,
 \item $\ast\=(\text{empty sum})$, and 
 \item if $m\=n$, then $m=n$ (in $M$).
\end{enumerate}
Let $B=\bpquot A\cR$ be a blueprint. A \emph{blue $B$-module} is a set $M$ together with a pre-addition $\cP$ and a \emph{$B$-action $B\times M\to M$}, which is a map $(b,m)\mapsto b.m$ that satisfies the following properties:
\begin{enumerate}
 \item $1.m=m$, $0.m=\ast$ and $a.\ast=\ast$,
 \item $(ab).m=a.(b.m)$, and
 \item $\sum a_i\=\sum b_j$ and $\sum m_k\=\sum n_l$ imply $\sum a_i.m_k\=\sum b_j.n_l$.
\end{enumerate}
A \emph{morphism of blue $B$-modules $M$ and $N$} is a map $f:M\to N$ such that
\begin{enumerate}
 \item $f(a.m)=a.f(m)$ for all $a\in B$ and $m\in M$ and 
 \item whenever $\sum m_i\=\sum n_j$ in $M$, then $\sum f(m_i) \= \sum f(n_j)$ in $N$. 
\end{enumerate}
This implies in particular that $f(\ast)=\ast$. We denote the category of blue $B$-modules by $\Mod B$. Note that in case of a ring $B$, every $B$-module is a blue $B$-module, but not vice versa.

\begin{lemma}
 The category $\Mod B$ is closed, complete and cocomplete. The trivial blue module $0=\{\ast\}$ is an initial and terminal object of $\Mod B$. 
\end{lemma}

\begin{proof}
 All arguments are essentially the same as in the case of $A$-sets. We refer to \cite[Section 2.2.1]{CLS12} for the facts that $\Mod B$ is closed and $0$ is initial and terminal. The construction of limits and colimits can be found in \cite[Prop.\ 2.13]{CLS12}.
 
 To conclude, we comment on structure of the morphism set $\Hom_B(M,N)$ as blue $B$-module. The scalar multiplication of a morphism $f:M\to N$ by an element $a\in B$ is given by $a.f:m\to a.f(m)$, and the pre-addition of $\Hom_B(M,N)$ consists of the relations $\sum f_i\=\sum g_j$ for which $\sum f_i(m)\=\sum g_j(m)$ for all $m\in M$. It is straight-forward to verify that these definitions satisfy the axioms of a blue $B$-module.
\end{proof}

\begin{lemma}\label{lemma: tensor products}
 The category $\Mod B$ has tensor products $M\otimes_B N$, which are characterized by the universal property that every bi-$B$-linear morphism $M\times N\to P$ factors through a unique $B$-linear map $M\otimes_B N\to P$. The canonical map $M\times N\to M\otimes_B N$ is surjective. The functor $\blanc\otimes_B M$ is left adjoint to $\Hom_B(M,\blanc)$. Together with the tensor product, $\Mod B$ is a symmetric monoidal category.
\end{lemma}

\begin{proof}
 The blue $B$-module $M\otimes_B N$ can be defined as the $\bpquot{M\times N}{\cP}$ where $\cP$ is the pre-addition generated by the relations
 \[
  (b.m,p) \= (m,b.p), \qquad \sum (m_i,p) \= \sum (n_j,p), \qquad \sum (m,p_k) \= \sum (m,q_l)
 \]
 for all $b\in B$, $m,m_i,n_j\in M$, $p,p_k,q_l\in N$ for which $\sum m_i\=\sum n_j$ holds in $M$ and $\sum p_k\=\sum q_l$ holds in $N$. It is easily verified that this defines a blue $B$-module that satisfies the desired properties; cf.\ section 2.2.3 of \cite{CLS12} for the analogous case of $A$-sets where $A$ is a monoid.
\end{proof}

Let $B$ be a blueprint. We denote the category of $B$-algebras by $\bp_B$ and its morphism sets by $\Hom_B(C,C')$. Let $\Fun$ be the blueprint $\bpgenquot{\{0,1\}}{\emptyset}$, which is an initial object in $\bp$. Then the association $(\Fun\to B)\mapsto B$ establishes an equivalence between $\bp_\Fun$ and $\bp$.

\begin{lemma}
 Let $B$ be a blueprint. Then the category $\bp_B$ is equivalent to the category of commutative monoids in $\Mod B$.
\end{lemma}

\begin{proof}
 A $B$-algebra $f:B\to C$ is a blue $B$-module w.r.t.\ the multiplication defined by $b.c=f(b)c$ for $b\in B$ and $c\in C$. The multiplication of $C$ turns $C$ into a commutative monoid in $\Mod B$. A morphism $C\to C'$ of $B$-algebras induces naturally a morphism of the associated commutative monoids in $\Mod B$. It is immediately verified that this functor is an equivalence of categories.
\end{proof}

%%%%%%%%%%%%%%%%%%%%%%%%%%%%%%%%%%%%%%%%%%%%%%%%%%%%%%%%%%%%%%%%%%%%%%%%%%%%%%%%%%%%%%%%%%%%%%%%%%%%%%%%%%%%%%%%%%%%%%%%%%%%%%%%%%%%%%%%%%%%%%%%%%

\section{Blueprint morphisms of finite presentation}
\label{section: blueprint morphisms of finite presentation}

Recall from section \ref{section: relative schemes} that a morphism $f:B\to C$ is of finite presentation if for all directed systems $\cD$, the canonical map
\[
 \Psi_\cD: \quad \colim \ \Hom_B (C,\cD) \quad \longrightarrow \quad \Hom_B(C,\colim \cD)
\]
is bijective. In this section, we characterize blueprint morphisms $f:B\to C$ of finite presentation in terms of the finiteness of certain sets of generators. 

Let $B=\bpquot A\cR$ be a blueprint. The we denote by $B[T_i]_{i\in I}$ the free blueprint over $B$ in the indeterminants $T_i$, cf.\ 1.12 of \cite{blueprints1}. Its elements are $\{0\}$ and all monomials $b\prod T_i^{n_i}$ with coefficients $b\in B-\{0\}$ where $n_i\geq0$ with $n_i=0$ for all but finitely many $i\in I$. The blueprint $B$ can be seen as the subset of all constants $b\prod T_i^0$, and the pre-addition of $B[T_i]$ is generated by the  image of $\cR$ in $B[T_i]$.

A $B$-algebra $f:B\to C$ is \emph{generated by a subset $\{b_i\}_{i\in I}$ of $C$} if there are for every element $c\in C$ finitely many $a_l\in B$ and not necessarily different indices in $i_l\in I$ such that $c\=\sum f(a_l)b_{i_l}$. A \emph{presentation of a $B$-algebra $f:B\to C$} is pair $(\fb,S)$ where $\fb=\{b_i\}_{i\in I}$ generates $f:B\to C$ and $S$ is a set of relations on the free $B$-algebra $B[T_i]_{i\in I}$ that satisfies the following property: for $\tilde B=\bpgenquot{B[T_i]}{S}$, there is a monomorphism $\tilde f:\tilde B\to C$ of $B$-algebras that sends $T_i$ to $b_i$ and the pre-addition of $C$ is generated by the image of the pre-addition of $\tilde B$. Note that the underlying monoid of $\tilde B$ is in general a proper quotient of $B[T_i]$.

A $B$-algebra $f:B\to C$ is \emph{algebraically of finite presentation} if there is a presentation $(\fb,S)$ of $f:B\to C$ with finite $\fb$ and finite $S$. We say that such a pair $(\fb,S)$ is a \emph{finite presentation} for $f:B\to C$.

\begin{prop}\label{prop: blueprint morphisms od finite presentation}
 A morphism $f:B\to C$ of blueprints is of finite presentation if and only if it is algebraically of finite presentation.
\end{prop}

\begin{proof}
 We unfold the definition of a finitely presented morphism of blueprints. Consider a directed system $\cD$ of $B$-algebras $D_i$ (where $i$ ranges through an index set $I$) and morphisms $g_{i,j}:D_i\to D_j$ (for a directed subset of indices $(i,j)\in I\times I$). Then the colimit of $\cD$ can be represented by the $B$-algebra
 \[
  \colim \cD \quad = \quad \coprod_{i\in I}\ D_i / \sim
 \]
 where the equivalence relation is generated by the relations $a_i\sim b_j$ between $a_i\in D_i$ and $b_j\in D_j$ for which there is a $k\in I$ such that $i,j\leq k$ and $g_{i,k}(a_i)=g_{j,k}(b_j)$. The pre-addition of $\colim \cD$ is generated by all relations of the form $\sum \bar a_k\=\sum \bar b_l$ for which $\sum a_k\=\sum b_l$ is a relation in some $D_i$ and where $\bar a_k$ and $\bar b_l$ are he images of $a_k$ and $b_l$, respectively, in $\colim \cD$. Let $\iota_i:D_i\to \colim \cD$ be the canonical morphisms.

 Similarly, an element of $\colim\Hom_B(C,\cD)$ can be represented by a morphism $\varphi_i:B\to D_i$ for some $i\in I$. The canonical map 
 \[
  \Psi_\cD: \quad \colim \ \Hom_B (C,\cD) \quad \longrightarrow \quad \Hom_B(C,\colim \cD)
 \]
 send the class of $\varphi_i:B\to D_i$ the morphism $\varphi=\iota_i\circ\varphi_i:C\to \colim  \cD$. Recall that $f:B\to C$ is of finite presentation if $\Psi_\cD$ is a bijection for all directed systems $\cD$.

 Assume $(\{b_1,\dotsc,b_n\},S)$ is a finite presentation for $f:B\to C$. Let $\cD$ be a directed system. 

 We show that $\Psi_\cD$ is injective. Let $\varphi_{i}$ and $\psi_j$ represent two elements of $\colim\Hom_B(C,\cD)$ such that $\varphi=\Psi_\cD(\varphi_i)=\Psi_\cD(\psi_j)=\psi$. Denote $\varphi_k=\varphi_i\circ g_{i,k}$ for $i\leq k$ and $\psi_k=\psi_j\circ g_{j,k}$ for $j\leq k$. Then there is an $k_l\in I$ and an $b_{k_l}\in D_{k_l}$ for every $l=1,\dotsc,n$ such that $\varphi_{k_l}(b_{k_l})=\psi_{k_l}(b_{k_l})$ in $D_{k_l}$. If $k\in J(k_1)\cap\dotsb\cap J(k_n)$, then $\varphi_{k}(b_k)=\psi_{k}(b_k)$ in $D_{k}$ for all $l=1,\dotsc,n$. Therefore, we obtain for an arbitrary element $c\=\sum f(a_l) b_{k_l}$ in $C$, the relation
 \[
  \varphi_k(c) \ \= \ \sum f(a_l) \varphi_k(b_{k_l}) \ \= \ \sum f(a_l) \psi_k(b_{k_l}) \ \= \ \psi_k(c),
 \]
 i.e.\ $\varphi_k=\psi_k$. This shows the injectivity of $\Psi_\cD$.

 We show that $\Psi_\cD$ is surjective. Let $\varphi:C\to \colim\cD$ be a morphism of $B$-algebras. Then for every relation $R$ in $\tilde f(S)$, there is an $i_R\in I$ and $c_l\in D_{i_R}$ such that $\varphi(b_l)=\iota_{i_R}(c_l)$ for $l=1,\dotsc,n$ and such that the $c_l$ satisfy the relation $\varphi_{i_R}(R)$. Since $S$ is finite, we can replace the $i_R$ by an $i\in I$ that is larger than all $i_R$ and can assume that there are elements $c_l$ in $D_i$ that satisfy all relations in $\varphi_i(\tilde f(S))$ and such that $\iota_i(c_l)=b_l$. This means that $\varphi: C\to \colim\cD$ factors into a morphism $\varphi_i:C\to D_i$, defined by $\varphi_i(b_l)=c_l$, followed by $\iota_i:D_i\to \colim\cD$. This establishes the surjectivity of $\Psi_\cD$ and shows that $f:B\to C$ is of finite presentation, which is one direction of the proposition.

 Assume that $\Psi_\cD$ is a bijection for every directed system $\cD$. Let $(\{b_i\}_{i\in I},S)$ be a presentation of $B\to C$ such that the cardinality of $I\cup S$ is minimal. We show that both $I$ and $S$ are finite.
 
 Define for every pair of finite subsets $J\subset I$ and $T\subset S$ such that all relations in $T$ involve only elements of $J$ the blueprint $D_{J,T}=\bpgenquot{B[b_i]_{i\in J}}{\cR_T}$ where $\cR_T$ is the pre-addition that is generated by $T$ and $\cR_B$. Then every $D_{J,T}$ is naturally a $B$-algebra and a pair of inclusions $J_1\subset J_2$ and $T_1\subset T_2$ yields a morphism $D_{J_1,T_1}\to D_{J_2,T_2}$ of $B$-algebras. This defines a directed system $\cD$ whose colimit $\colim\cD$ is $C$. Since $\Psi_\cD$ is bijective, the identity morphism $\id:C\to C=\colim\cD$ comes from an element of $\colim\Hom_B(C,\cD)$, represented by some morphism $\varphi_{J,T}:C\to D_{J,T}$. This means that there are a finite subset $J$ of $I$ and a finite subset $T$ of $S$ such that $\id:C\to C$ factorizes into $\varphi_{J,T}$, followed by $\iota_{J,T}:D_{J,T}\to C$. By the minimality of $I$ and $S$, this can only be the case if both $J=I$ and $T=S$, i.e.\ $I$ and $S$ are finite. This finishes the proof of the 
proposition.
\end{proof}

%%%%%%%%%%%%%%%%%%%%%%%%%%%%%%%%%%%%%%%%%%%%%%%%%%%%%%%%%%%%%%%%%%%%%%%%%%%%%%%%%%%%%%%%%%%%%%%%%%%%%%%%%%%%%%%%%%%%%%%%%%%%%%%%%%%%%%%%%%%%%%%%%%

\section{Flat morphisms}
\label{section: flat morphisms}

In this section, we show that a flat morphisms of blueprints coincide with localizations. 

We recall the definition of a localization of a blueprint at a multiplicative subset. Let $B=\bpquot A\cR$ be a blueprint. Let $S$ be a \emph{multiplicative set} in $B$, i.e.\ a subset of $B$ that contains $1$ and $ab$ for all $a,b\in S$. We define $S^{-1}A$ as the quotient of $A\times S$ by the equivalence relation $\sim$ given by $(a,s)\sim(a',s')$ if and only if there is a $t \in S$ such that $tsa'=ts'a$. We write $\frac as$ for the equivalence class of $(a,s)$ in $S^{-1}A$. We define $S^{-1}\cR$ as the set 
\[
 S^{-1}\cR \quad = \quad \Bigr\{\ \sum\frac{a_i}{s_i}\=\sum\frac{b_j}{r_j}\ \Bigl| \ \exists\, t\in S\text{ such that }\sum ts^ia_i\=\sum tr^jb_j\ \Bigl\} 
\]
where
\[
 \quad s^i=\prod_{k\neq i}s_k\cdot\prod_j r_j\quad \text{ and }\quad r^j=\prod_is_i\cdot\prod_{l\neq j} r_l.
\]
Then $S^{-1}A$ is a monoid (with the multiplication inherited from $A\times S$) and that $S^{-1}\cR$ is a pre-addition for $S^{-1}A$. We define the \emph{localization of $B$ at $S$} as the blueprint $S^{-1}B=\bpquot{S^{-1}A}{S^{-1}\cR}$.

 The association $a\mapsto \frac a1$ defines an epimorphism $B\to S^{-1}B$. It satisfies the universal property that every morphism $f:B\to C$ that maps $S$ to the units of $C$ factors uniquely through $B\to S^{-1}B$. If $S=\{h^i\}_{i\geq0}$ is generated by some $h\in B$, then we denote $S^{-1}B$ by $B[h^{-1}]$.

 Given a blue $B$-module $M$ and a multiplicative subset $S$ of $B$, we define $S^{-1}M$ as the following blue $B$-module. Its underlying set is the quotient of $M\times S$ by the equivalence relation $\sim$ defined by $(m,s)\sim(m',s')$ if and only if there is a $t\in S$ such that $ts.m'=ts'.m$. We denote by $\frac ms$ the equivalence class of $(m,s)$ and denote by $f:M\to S^{-1}M$ the canonical map that sends $m$ to $\frac m1$. The pre-addition of $S^{-1}M$ is generated by $f(\cP)$ where $\cP$ is the pre-addition of $M$. With this $f:M\to S^{-1}M$ is a morphism of blue $B$-modules, and $S^{-1}M$ is naturally a blue $S^{-1}B$-module.

 We say that a morphism $f:B\to C$ of blueprints is a \emph{localization} if there is a multiplicative set $S$ in $B$ and an isomorphism $g: C\to S^{-1}B$ such that $g\circ f:B\to S^{-1}B$ equals the localization map $a\mapsto \frac a1$. Recall that a morphism $f:B\to C$ is flat if $\blanc\otimes_BC:\Mod B\to \Mod C$ commutes with finite limits and colimits.

\begin{prop}\label{prop: flat equals localization}
 A morphism $f:B\to C$ of blueprints is flat if and only if it is a localization.
\end{prop}

\begin{proof}
 Since it is not surprising that localizations are flat, we restrict ourselves to an outline of this direction of the proof. Let $S$ be a multiplicative subset of $B$. Since $\blanc\otimes_BS^{-1}B$ is left adjoint to $\Hom_B(S^{-1}B,\blanc)$, it commutes with colimits. It is easily verified that $\blanc\otimes_BS^{-1}B$ commutes with finite limits (cf.\ \cite[Prop.\ 2.24]{CLS12} for the case of a monoid $B$). Therefore $B\to S^{-1}B$ is flat.

 For two blue $B$-modules $M$ and $N$, the universal morphism $M\times N\to M\otimes_BN$ is a surjection between the underlying sets, see Lemma \ref{lemma: tensor products}. This means that every element of $M\otimes_BN$ can be represented as $m\otimes n$ with $m\in M$ and $n\in N$. This is the basic property that allows us to show that every flat morphism is a localization.

 Assume that $f:B\to C$ is flat. Define $S=f^{-1}(C^\times)$, which is a multiplicative subset of $B$. By the universal property of the localization, the morphism $f:B\to C$ factors into $B\to S^{-1}B$ and a unique morphism $f_S:S^{-1}B \to C$. The proof is completed once we have shown that $f_S$ is an isomorphism.

 We show that $f_S$ is surjective. For $b\in B$ and $c\in C$, we write $b.c=f(b)c$. By the flatness of $f:B\to C$, the morphism 
 \[
  \begin{array}{cccccc}
   \Phi:& (B\times B)\otimes_BC & \longrightarrow & (B\otimes_BC)\times (B\otimes_BC) & = & C\times C \\
        &  (b_1,b_2)\otimes c   & \longmapsto     &  (b_1\otimes c,b_2\otimes c)      & = & (b_1.c,b_2.c)
  \end{array}
 \]
 is an isomorphism. This means that we find for every $d\in C$ elements $b_1,b_2\in B$ and $c\in C$ such that $(d,1)=(b_1.c,b_2.c)$. Therefore $b_2.d=b_1b_2.c=b_1.1$ is in $f(B)$ and $f(b_2)\in C^\times$. This means that $d=f_S(\frac{b_1}{b_2})$ is in the image of $f_S$, which shows that $f_S$ is surjective.

 If $A$ is the underlying monoid of $S^{-1}B$ and $\cR_B$ is its pre-addition, then we can represent $C$ as $\bpquot A{\cR_C}$, which is not necessarily a proper representation. If we show that $f_S$ induces a bijection between $\cR_B$ and $\cR_C$, then it follows that $f_S$ is an isomorphism. 

 To do so, consider a relation $\sum a_i\=\sum b_j$ between elements $a_i,b_j\in A$ in $\cR_C$. Consider the inclusion $B\to B^+$ of $B$ into its associated semiring. Define $a=\sum a_i$ and $b=\sum b_j$ in $S^{-1}B^+$ and consider the two morphisms
 \[
  \xymatrix{S^{-1}B^+  \ar@<0,5ex>[rr]^{f_a} \ar@<-0,5ex>[rr]_{f_b} && S^{-1}B^+} 
 \]
 of blue $B$-modules given by $f_a(c)=ac$ and $f_b(c)=bc$. Since $B\to S^{-1}B$ is an epimorphism and $S^{-1}B\to C$ is surjective, we have $S^{-1}B^+\otimes_BC=S^{-1}B^+\otimes_{S^{-1}B}C=C^+$. Therefore the base change $\blanc\otimes_BC$ yields the morphisms 
 \[
  \xymatrix{C^+  \ar@<0,5ex>[rr]^{f_a\otimes_BC} \ar@<-0,5ex>[rr]_{f_b\otimes_BC} && C^+},
 \]
 which are the same since $a=\sum a_i =\sum b_j=b$ in $C^+$. Since $f:B\to C$ is flat, we have 
 \[
  \eq\bigl(f_a,f_b\bigr)\otimes_BC \quad = \quad \eq\bigl(f_a\otimes_BC,f_b\otimes_BC\bigr) \quad = \quad C^+.
 \]
 Therefore, there exists a $c\in \eq(f_a,f_b)$ such that $f(c)\in (C^+)^\times$ and $c\otimes f(c)^{-1}=1$ in $C^+$. By the definition of $S$, $c$ is invertible in $S^{-1}B^+$. Therefore the equality $ac=f_{a}(c)=f_{b}(c)=bc$ implies
 \[
  a \quad = \quad acc^{-1}\quad = \quad c^{-1}f_a(c)\quad = \quad c^{-1}f_b(c)\quad = \quad bcc^{-1} \quad = \quad b
 \]
  in $S^{-1}B^+$. This means that $\sum a_i\=\sum b_j$ in $S^{-1}B$, which was to be shown.
\end{proof}

\begin{cor}
 Every flat morphism of blueprints is an epimorphism. \qed
\end{cor}

We say that a morphism $f:B\to C$ is a \emph{finite localization} if it is isomorphic to a localization $B\to S^{-1}B$ at a finitely generated multiplicative subset $S=\{h_1^{e_1}\dotsb h_n^{e_n}|e_i\geq 0\}$ of $B$. Note that in this case, $S^{-1}B=B[h^{-1}]$ for $h=\prod h_i$.

\begin{cor}\label{cor:flat epi ofp euqals finite localization}
 A morphism of blueprints is a flat epimorphism of finite presentation if and only if it is a finite localization.
\end{cor}

\begin{proof}
 We know that a flat epimorphism is the same as a localization. Let $B\to S^{-1}B$ be a finite localization, i.e. $S$ is generated by finitely many elements $h_1,\dotsc,h_n$. Then $(\{T_1,\dotsc,T_n\},\{T_1h_1\=1,\dotsc,T_nh_n\=1\})$ is a finite presentation for $B\to S^{-1}B$. Conversely, if $B\to S^{-1}B$ is a localization with a finite presentation $(\{T_1\dotsc,T_n\},R)$, then $S$ is finitely generated by those $T_i$ that are invertible in $S^{-1}B$.
\end{proof}

%%%%%%%%%%%%%%%%%%%%%%%%%%%%%%%%%%%%%%%%%%%%%%%%%%%%%%%%%%%%%%%%%%%%%%%%%%%%%%%%%%%%%%%%%%%%%%%%%%%%%%%%%%%%%%%%%%%%%%%%%%%%%%%%%%%%%%%%%%%%%%%%%%

\section{Geometric blue schemes}
\label{section: geometric blue schemes}

In this section, we recall the definition of a blue scheme as a blueprinted space that is locally isomorphic to the spectra of blueprints, as introduced in \cite{blueprints1}. In order to contrast blue schemes with relative blue schemes, as considered in section \ref{section: relative schemes}, we call blue schemes also geometric blue schemes in the following. In order to make the terminology more coherent with the other literature on semirings where an ideal is a submodule of the semiring (cf.\ section \ref{section: semiring schemes}), we will use the term $k$-ideal for what is called an ideal in \cite{blueprints1}.

A \emph{$k$-ideal} of a blueprint $B$ is a subset $I$ of $B$ such that $0\in I$, $IB=I$ and $c\in I$ whenever $c+\sum a_i\=\sum b_j$ with $a_i,b_j\in I$. For every $k$-ideal $I$ of $B$, there is a universal morphism $f:B\to B/I$ of blueprints with $f^{-1}(0)=I$, and we call $B/I$ the \emph{quotient of $B$ by $I$}. A $k$-ideal $I$ is \emph{maximal} if $I\subsetneq B$ and if there is there is no strictly larger $k$-ideal $J\subsetneq B$. 

A \emph{blueprinted space} is a topological space $X$ together with a sheaf $\cO_X$ with values in $\bp$. A \emph{morphism of blueprinted spaces $X$ and $Y$} consists of a continuous map $\varphi:X\to Y$ between the underlying topological spaces and a morphism $\varphi^\#:\cO_Y\to \varphi_\ast\cO_X$ of sheaves such that the induced morphisms $\varphi^\#_x:\cO_{Y,\varphi(x)}\to \cO_{X,x}$ of stalks are \emph{local}, i.e.\ it maps noninvertible elements to noninvertible elements. 

\begin{rem}
 Note that we consider two different ways of associating a blueprinted space with a blueprint, which are the \emph{geometric spectrum}, as defined below, and the \emph{subcanonical spectrum}, as defined in \ref{section: the blueprinted space of a subcanonical blue scheme}. For either spectrum, the blueprinted space will be local \emph{in the appropriate sense}, but these locality conditions do not agree. Therefore we avoid the definition of a ``locally blueprinted space'' and circumvent this by a more general definition of local morphisms between the stalks that applies to both notions of spectra in the correct way.
\end{rem}

A \emph{prime $k$-ideal} of a blueprint $B$ is a $k$-ideal $\fp$ such that $S=B-\fp$ is a multiplicative subset of $B$. We endow the set $X$ of all prime $k$-ideals of $B$ with the topology that is generated by the subsets $U_h=\{\fp\in X|h\notin\fp\}$ where $h\in B$. Note that $U_g\cap U_h=U_{gh}$, which implies that $\{U_h|h\in B\}$ forms a basis for the topology of $X$. A \emph{covering family for $X$} is a collection of open subsets whose union equals $X$. The \emph{structure sheaf $\cO_X$} is defined as the sheaf that associates with each open subset $U$ of $X$ the set of locally representable sections $s:U\to \coprod_{\fp\in U}B_\fp$, i.e.\ there is a covering family $\{U_{h_i}\}$ of $X$ and elements $s_i\in B[h_i^{-1}]$ such that $s(\fp)=s_i$ in $B_\fp$ for all $i$ and $\fp\in U_{h_i}$. Note that each set $\cO_X(U)$ of local sections comes with the natural structure of a blueprint. The \emph{spectrum $\Spec B$ of $B$} is the topological space $X$ together with the structure sheaf $\cO_X$. 

The spectrum of a blueprint is a blueprinted space. The stalk $\cO_{X,x}$ at a point $x$ of $X$ is a \emph{local blueprint}, i.e.\ $\fm_x=\cO_{X,x}-\cO_{X,x}^\times$ is a maximal $k$-ideal. The \emph{residue field at $x$} is the quotient $\cO_{X,x}/\fm_x$, which is a \emph{blue field}, i.e.\ a blueprint with $0\neq 1$ whose only noninvertible element is $0$.

A morphism $f:B\to C$ of blueprints defines naturally a morphism $f^\ast:\Spec C\to \Spec B$ between the spectra of the blueprints. Thus $\Spec$ defines a functor from $\bp$ to the category $\bpspaces$ of blueprinted spaces. Conversely, taking global sections $\Gamma(X,\cO_X)$ defines a functor from $\bpspaces$ to $\bp$. We obtain an endofunctor on blueprints that sends $B$ to $\Gamma B=\Gamma(X,\cO_X)$ where $X=\Spec B$. 

The difficulty in comparing blue schemes with schemes relative to the category of blue $\Fun$-modules lies in the fact that the functor $\Spec:\bp\to\bpspaces$ is not fully faithful, and that the canonical morphism $\sigma:B\to\Gamma B$, called the \emph{globalization of $B$}, is in general not an isomorphism. We call $B$ \emph{global} if $\sigma$ is an isomorphism. We have the following results; see \cite[Thm.\ 3.12 and Cor.\ 3.13]{blueprints1}.

\begin{thm}\label{thm: globalization induces an isomorphism between spectra}
 For every blueprint $B$, $\sigma:B\to \Gamma B$ defines an isomorphism $\sigma^\ast:\Spec\Gamma B\to\Spec B$. Consequently, $\Gamma B$ is a global blueprint and every morphism $f:B\to C$ into a global blueprint $C=\Gamma C$ factors uniquely through $\sigma:B\to \Gamma B$.
\end{thm}

Let $\Gamma\bp$ be the full subcategory of $\bp$ whose objects are global blueprints. By means of Theorem \ref{thm: globalization induces an isomorphism between spectra}, the globalization defines a functor $\Gamma:\bp\to\Gamma\bp$ satisfying the following property.

\begin{cor}\label{cor: a right adjoint and left inverse to the globalization functor}
 The embedding $\iota:\Gamma\bp\to\bp$ as a subcategory is right adjoint and left inverse to $\Gamma:\bp\to\Gamma\bp$.\qed
\end{cor}

If we denote the full subcategory of global blueprints in $\bp$ by $\Gamma\bp$, then the restriction of $\Spec$ to $\Gamma\bp$ is a fully faithful embedding into $\bpspaces$, and $\Gamma\circ\Spec$ is isomorphic to the identity functor of $\Gamma\bp$. Examples of global blueprints are local blueprints, due to the lack of nontrivial coverings of their spectra, monoids and rings, cf.\ \cite[section 3.2]{blueprints1}.  Examples of blueprint that are not global can be found in \cite[Ex.\ 3.8]{blueprints1}, in Examples \ref{ex: semiring that is not global} and \ref{ex: finite localizations do not commute with globalization} and in Remark \ref{rem: algebraically presented blue schemes}.

Let $\Aff_\Fun^\geo$ be the essential image of $\Spec$, which is a full subcategory of $\bpspaces$. We endow $\Aff_\Fun^\geo$ with the Grothendieck topology that is generated by covering families of the form $\{\varphi_i:\Spec B[h_i^{-1}]\to \Spec B\}$ such that the underlying topological space of $\Spec B$ is covered by the images of the $\varphi_i$. We can define blue schemes in terms of the following characterization. 

\begin{thm} \label{thm: blue schemes as colimits of affine presentations}
 A blueprinted space $X$ is a blue scheme if and only if there is an affine presentation $\cU$ in $\Aff_\Fun^\geo$ such that $X\simeq\colim \cU$ in $\bpspaces$.
\end{thm}

\begin{proof}
 Without recalling all definitions from \cite{blueprints1}, we sketch the proof of this fact. Let $X$ be a blue scheme and $\cV$ an affine open covering. We define $\cU$ as follows. The maximal elements are $\cU_{\max}=\cV$. For every pair of distinct $U$ and $V$ in $\cU_{\max}$, we let $\cU_{U,V}$ be the set of all affine open subschemes of $U\cap V$. We define $\cU$ as the disjoint union of $\cU_{\max}$ with the sets $\cU_{U,V}$ where $U$ and $V$ range through $\cU_{\max}$, and the morphisms of $\cU$ are the inclusions $W\to U$ and $W\to V$ for $U,V\in\cU_{\max}$ and $W\in\cU_{U,V}$. Then $\cU$ is a commutative diagram of open immersions such that $X=\colim\cU$. Its maximal elements are $\cU_{\max}$, and $\cU$ is monodromy-free since maximal elements of $\cU$ are open subschemes of $X$. Thus $\cU$ is an affine presentation and, indeed, an affine atlas.

 Conversely, if $X$ is the colimit of an affine presentation $\cU$, then $X$ is covered by the images of the maximal elements in $\cU$. The maximal elements of $\cU$ are isomorphic to their image in $X$, as can be seen as follows. Note that this means in particular that $X$ is covered by the maximal elements in $\cU$ as a topological space.
 
 Thus as a set, $X$ equals the disjoint union $\coprod U$ over all affine blue schemes $U$ in $\cU$, modulo the identifications $x\sim \varphi(x)$ for every open immersion $\varphi:U\to V$ in $\cU$ and every point $x$ in $U$. If $x\sim y$ for two elements $x$ and $y$ of $U$, then their must be a path 
 \[
   \xymatrix@C=4pc{U=V_0 \ar@{-}[r]^{\varphi_1} & V_1 \ar@{-}[r]^{\varphi_2} & \dotsb \ar@{-}[r]^{\varphi_{n-1}} & V_{n-1} \ar@{-}[r]^{\varphi_n} & V_n=U}
 \]
 from $U$ to itself and elements $z_i\in V_i$ such that $x=z_0\sim z_1\sim\dotsb\sim z_{n-1}\sim z_n=y$. We denote this path by $\cV$. Let $\cW$ be the diagram of the induced morphisms between the stalks of $x,z_1,\dotsc,z_{n-1},y$ and $W$ its limit. Then $W$ comes with a family of morphism $\{\psi_i:W\to V_i\}$ where the image of $\psi_i$ contains $z_i$. Equivalently, we have a morphism $\Psi:W\to \lim\cV$. Since $\cU$ is monodromy-free, $\Psi$ factors uniquely through a morphism $W\to \lim(\cV\!\!\!\begin{array}{c}\to\vspace{-8.5pt}\\ \to\end{array} U \bigr)$. It follows that $x=y$.

% Let $\cV'$ be the trivial path from $U$ to $U$ and $\psi_0:W\to V_0=U=\lim\cV'$ the map with image $x$. This yields the diagram
% \[
%  \xymatrix@R=0.5pc@C=5pc{   &  \lim\cV \ar[r]^(0.45){\varphi_0}\ar[rdd]^(0.4){\varphi_n}|\hole & U \\  W \ar[ur]^{\psi} \ar[dr]_{\psi_0} \\ & U\ar[uur]_(0.4){\id}\ar[r]_(0.45){\id} & U}
% \]
% where $\psi=\lim\psi_i$ is the induced morphism and where $\varphi_0:\lim\cV\to V_0=U$ and $\varphi_n:\lim\cV\to V_n=U$ are the canonical projections. By definition of this diagram, the upper square commutes. Since $\cU$ is monodromy-free, we conclude that also the lower diagram commutes. This means that $\{x\}=\im (\varphi_n\circ\psi)=\im (\id\circ \psi_0)=\{y\}$, i.e.\ $x=y$.
 
 This shows that every affine scheme $U$ in $\cU$ injects into $X$ as a set. Since all morphisms of $\cU$ are open immersions and therefore homeomorphisms onto their images, the injections $U\to X$ are also homeomorphisms onto their image. 
 
 We are left with showing that the restrictions of the structure sheaf of $X$ to $U$ corresponds to the structure sheaf of $U$ for all $U$ in $\cU$. We can employ the same formal argument as for the injectivity of the maps $U\to X$, though this time all arrows are reversed. We leave the details to the reader.
\end{proof}

We denote the full subcategory of $\bpspaces$ whose objects are geometric blue schemes by $\Sch_\Fun^\geo$.

\begin{prop}\label{prop: properties 1 to 6 for blue schemes}
 The embeddings of $\Aff_\Fun^\geo\to\Sch_\Fun^\geo\to\bpspaces$ satisfy Hypothesis \ref{hypothesis}.
\end{prop}

\begin{proof}
 Property \eqref{part1} of Hypothesis \ref{hypothesis} follows from the construction of the fibre product of blue schemes, cf.\ \cite[Prop.\ 3.27]{blueprints1}. Property \eqref{part2} is established by Proposition \ref{thm: blue schemes as colimits of affine presentations}. Properties \eqref{part4} and \eqref{part5} follow easily from the facts that open subschemes are completely determined by their underlying topological space and that taking colimits of affine presentations commutes with the forgetful functor to topological spaces. Property \eqref{part6} is \cite[Thm.\ 3.23]{blueprints1}, and \eqref{part3} follows from \eqref{part6}.
\end{proof}

%%%%%%%%%%%%%%%%%%%%%%%%%%%%%%%%%%%%%%%%%%%%%%%%%%%%%%%%%%%%%%%%%%%%%%%%%%%%%%%%%%%%%%%%%%%%%%%%%%%%%%%%%%%%%%%%%%%%%%%%%%%%%%%%%%%%%%%%%%%%%%%%%%

\section{Subcanonical blue schemes}
\label{section: subcanonical blue schemes}

We can remedy the discrepancy between $B$ and its globalization $\Gamma B$ by applying To\"en and Vaqui\'e's machinery to the category $\Mod\Fun$ of blue $\Fun$-modules. This implies the desired fact that every blueprint represents a sheaf. We call the resulting objects \emph{subcanonical blue schemes} to contrast them with geometrical blue schemes. This name will be justified in Lemma \ref{lemma: subcanonical coverings of blueprints are TV-coverings}.

Let $\Aff^{\can}_\Fun=\bp^\op$ be the category of affine schemes relative to $\Mod\Fun$. We call its objects \emph{affine subcanonical blue schemes}. We call the Grothendieck topology as defined in section \ref{section: relative schemes} the \emph{subcanonical Zariski topology} on $\Aff^\can_\Fun$.

We say that a morphism $\spec C\to \spec B$ is a \emph{finite localization} if $B\to C$ is a finite localization, and we say that a family of morphisms $\{U_i\to X\}$ is \emph{subcanonical} if it is a covering family in the canonical topology for $\Aff^\can_\Fun$.

\begin{lemma}\label{lemma: subcanonical coverings of blueprints are TV-coverings}
 A family $\{U_i\to X\}$ of finite localizations in $\Aff^{\can}_\Fun$ is a covering family in the subcanonical Zariski topology if and only if it is subcanonical.
\end{lemma}

\begin{proof}
 The proof of Lemma \ref{lemma: subcanonical coverings of semirings are TV-coverings} applies literally to this case, with the only variation that we define $B[M]$ as $\bigvee_{i\geq0} \bigl(M^{\otimes i} \,\bigl/\, \gen{m_1\otimes\dotsb\otimes m_i-m_{\sigma(1)}\otimes\dotsb\otimes m_{\sigma(i)}|\sigma\in S_i}\bigr)$.
\end{proof}

Lemma \ref{lemma: subcanonical coverings of blueprints are TV-coverings} tells us that the relative Zariski topology on $\Aff^\can_\Fun$ is the finest subcanonical topology that is generated by families of finite localizations. The following observation, however, shows that this topology is very coarse.

We call an object $X$ of a site \emph{geometrically local} if every covering family $\{U_i\to X\}$ contains an isomorphism. For example, the spectrum of a ring $B$ is geometrically local if and only if $B$ is a local ring. More generally, this holds for a blueprint $B$: its spectrum $\Spec B$ is geometrically local in $\Aff_\Fun^\geo$ if and only if $B$ is a local blueprint.

\begin{prop}\label{prop: affine subcanonical blue schemes are geometrically local}
 Every affine subcanonical blue scheme is geometrically local in $\Aff^\can_\Fun$.
\end{prop}

\begin{proof}
 By the definition of the relative Zariski topology, it is enough to prove the lemma for a covering family of the form $\{\spec B[h_i^{-1}]\to \spec B\}$.
 
 Consider the subset $M$ of all non-invertible elements of $B$, endowed with all additive relations that hold in $B$. Then the inclusion $\iota:M\to B$ is a morphism of blue $B$-modules. 
 
 Since tensor products preserve injections, $\iota[h_i^{-1}]:M[h_i^{-1}]=M\otimes_BB[h_i^{-1}]\to B[h_i^{-1}]$ is injective. If $h_i$ is not invertible, then $\iota[h_i^{-1}]$ is also surjective since $1=\frac{h_i}{h_i}$ is in $M[h_i^{-1}]$. In this case, every relation $\sum m_k\=\sum n_l$ in $B$ can be re-obtained from the relation $\sum h_im_k\=\sum h_i n_l$ in $M$. This shows that $\iota[h_i^{-1}]$ is an isomorphism of blue $B[h_i^{-1}]$-modules if $h_i$ is not invertible.
 
 Since $\iota:M\to B$ is not an isomorphism, there must be some $i$ such that $\iota[h_i^{-1}]$ is not an isomorphism, i.e.\ $h_i$ must be a unit of $B$. But the localization $B\to B[h_i^{-1}]$ at a unit $h_i$ is an isomorphism, which proves the lemma.
\end{proof}

This lemma implies that the category $\Sch^\can_\Fun$ of subcanonical blue schemes differs drastically from the category of geometric blue schemes. For example, we have for every subcanonical blue scheme $X$ with affine presentation $\cU$ in $\Aff^\can_\Fun$ that $X(B)=\colim \cU(B)$ for any blueprint $B$, which is not true for geometric blue schemes unless $\Spec B$ is geometrically local in $\Aff_\Fun^\geo$.

Let $\Gamma^\can:\Aff_\Fun^\can\to\bp$ and $\Gamma^\geo:\Aff_\Fun^\geo\to\bp$ denote the respective global section functors.

\begin{thm}\label{thm: the geometric Zariski topology is finer than the subcanonical Zariski topology}
 The functor $\Spec\circ\Gamma^\can: \Aff_\Fun^\can\to\Aff_\Fun^\geo$ sends covering families to covering families, and $\spec\circ\Gamma^\geo:\Aff_\Fun^\geo\to\Aff_\Fun^\can$ is its left adjoint. 
\end{thm}

\begin{proof}
 That $\Spec\circ\Gamma^\can$ sends covering families to covering families is an immediate consequence of Corollary \ref{cor:flat epi ofp euqals finite localization} and Lemma \ref{lemma: subcanonical coverings of blueprints are TV-coverings}. The latter claim follows from the corresponding fact for the dual categories: the inclusion $\iota:\Gamma\bp\to \bp$ is right adjoint to $\Gamma:\bp\to\Gamma\bp$, cf.\ Corollary \ref{cor: a right adjoint and left inverse to the globalization functor}.
\end{proof}

\begin{cor}\label{cor: functor from subcanonical to geometric blue schemes}
 The functor $\Spec\circ\Gamma^\can:\Aff_\Fun^\can\to\Aff_\Fun^\geo$ extends to a functor $\cG:\Sch_\Fun^\can\to\Sch_\Fun^\geo$ that sends a colimit of an affine presentation to the colimit of the same affine presentation.
\end{cor}

\begin{proof}
 Since left adjoints preserve limits, Proposition \ref{thm: the geometric Zariski topology is finer than the subcanonical Zariski topology} verifies the hypotheses of Lemma \ref{lemma: extension of a morphism of sites to a functor between schemes} for $\Spec\circ\Gamma^\can:\Aff_\Fun^\can\to\Aff_\Fun^\geo$. This yields the claim of the corollary at once.
\end{proof}

This completes the construction of the functor $\cG$. Since $\cG(\spec B)=\Spec B$, we can rephrase the content of Corollary \ref{cor: functor from subcanonical to geometric blue schemes} by saying that $\cG$ extends the identity functor on $\bp$, as claimed in part \eqref{A3} of Theorem A. Before we turn to the construction of the functor $\cG^+$, we derive a description of subcanonical blue schemes as blueprinted spaces in the following section.

%%%%%%%%%%%%%%%%%%%%%%%%%%%%%%%%%%%%%%%%%%%%%%%%%%%%%%%%%%%%%%%%%%%%%%%%%%%%%%%%%%%%%%%%%%%%%%%%%%%%%%%%%%%%%%%%%%%%%%%%%%%%%%%%%%%%%%%%%%%%%%%%%%

\section{The blueprinted space of a subcanonical blue scheme}
\label{section: the blueprinted space of a subcanonical blue scheme}

In this section, we show how to realize subcanonical blue schemes as blueprinted spaces. The key lemma for this interpretation is the following characterization of the locale of all open subschemes of an affine subcanonical blue scheme $X=\spec B$.

We can consider the underlying monoid $A$ of $B$ as the blueprint $\bpgenquot{A}{\emptyset}$. The points of the blueprinted space $\Spec A$ are the prime $k$-ideals of $A$. Since the preaddition of $A$ is trivial, a prime $k$-ideal is a subsets $\fp$ of $A$ with $0\in\fp$, $\fp A=\fp$ and whose complement $A-\fp$ is a multiplicative set.

\begin{lemma}\label{lemma: locale of an affine subcanonical blue scheme}
 Let $B$ be a blueprint with underlying monoid $A$. Then the locale of all open subschemes of $\spec B$ is canonically isomorphic to the locale of all open subsets of $\Spec A$.
\end{lemma}

\begin{proof}
 Let $F$ be the locale of all open subschemes $U$ of $\spec B$ and $E$ the locale of all open subschemes of $\Spec A$. We know that the principal open subschemes $V_h=\spec B[h^{-1}]$ of $\spec B$ form a basis of $F$ and that the principal open subschemes $U_h=\Spec A[h^{-1}]$ of $\Spec A$ form a basis of $E$, where in each case $h$ varies through all elements of $B$. This means that in both cases the open subschemes are unions of principal opens.
 
 By Lemma \ref{prop: affine subcanonical blue schemes are geometrically local}, we know that $V_g=\bigvee_{i\in I} V_{h_i}$ if and only $V_g=V_{h_i}$ for some $i$. The same holds true for the principal opens $U_h$ of $\Spec A$ since every $U_h$ contains a unique maximal point, which is the prime $k$-ideal $\fp=\{a\in A|h\notin aA\}$ of $A$.
 
 This shows that the association $\bigvee V_{h_i}\mapsto \bigcup U_{h_i}$ is a bijection between $F$ and $E$. This bijection is an isomorphism of locales since it respects the respective partial orders: we have $V_g\leq V_h$ if and only if $g\in hB$, which is also equivalent to $U_g\subset U_h$.
\end{proof}

\begin{rem}
 This characterization of the underlying topological space is similar to Marty's result in \cite{Marty07}, which says that under certain assumptions on the module category $\cC$, the locale of an affine scheme $\spec B$ relative to $\cC$ is determined by the submodules of $B$. These conditions are not satisfied by the category of blue $B$-modules for any non-trivial $B$: though $\Mod B$ is a \emph{relative context} in Marty's terminology, it fails to be \emph{strong}, i.e.\ the functor $\Hom_B(B,-):\Mod B\to \Sets$ does not reflect isomorphisms. 
 
 Still, one can draw an analogy between Marty's result and ours. We say that a submodule $M$ of a blueprint $B$ is \emph{full} if the preaddition of $M$ is the restriction of the preaddition of $B$ to $M$. A full submodule of $B$ is determined by its underlying set, and a subset $M$ of $B$ carries the structure of a full submodule of $B$ if and only if $M$ contains $0$ and $BM=M$. In other words, the full submodules of $B$ correspond to the $k$-ideals of the underlying monoid $A$ of $B$. Therefore Lemma \ref{lemma: locale of an affine subcanonical blue scheme} states that the space of all full submodules of $B$ is the underlying topological space of $\spec B$.
\end{rem}

We define the blueprinted space $X$ of $\spec B$ as the underlying topological space of $\Spec A$, together with the structure sheaf $\cO_X$ that associates with a principal open subset $U_h=\{\fp|h\notin\fp\}$ the blueprint $B[h^{-1}]$. Lemma \ref{lemma: locale of an affine subcanonical blue scheme} guarantees that this definition indeed extends to a sheaf on the topological space $X$. Similar to the case of geometric blue schemes, the stalk $\cO_{X,\fp}$ at a prime $k$-ideal $\fp$ of $A$ is the blueprint $B_\fp=S^{-1}B$ where $S=B-\fp$ is the complement of $\fp$.

Given an arbitrary subcanonical blue scheme with affine presentation $\cU$, we obtain a diagram of associated blueprinted spaces. We define the associated blueprinted space as its colimit. This definition does not depend on the choice of $\cU$. We obtain a fully faithful embedding of the category of subcanonical blue schemes into the category of blueprinted spaces. Moreover, the embeddings $\Aff_\Fun^\can\to\Sch_\Fun^\can\to\bpspaces$ satisfy Hypothesis \ref{hypothesis}.

%%%%%%%%%%%%%%%%%%%%%%%%%%%%%%%%%%%%%%%%%%%%%%%%%%%%%%%%%%%%%%%%%%%%%%%%%%%%%%%%%%%%%%%%%%%%%%%%%%%%%%%%%%%%%%%%%%%%%%%%%%%%%%%%%%%%%%%%%%%%%%%%%%

\section{Semiring schemes}
\label{section: semiring schemes}

In this section, we show that there is a natural association from semiring schemes in the sense of To\"en and Vaqui\'e to semiring schemes as objects of the category of geometric blue schemes.

Recall that we consider the category of semirings embedded into the category of blue\-prints, and that we call blueprints in the essential image of this embedding semirings, by abuse of language.

A \emph{geometric semiring scheme} is a geometric blue scheme such that for all open subsets $U$ of $X$, the blueprint $\cO_X(U)$ is a semiring. We denote the full subcategory of $\Aff^\geo_\Fun$ whose objects are affine semiring schemes by $\Aff_\N^{+,\geo}$. It is a site with respect to the restriction of the Zariski topology of $\Aff_\Fun^\geo$.

We define $\Mod^+\N$ as the category of commutative semigroups with a neutral element, and denote by $\spec:\SRings\to \Aff_\N^{+,\can}$ the anti-equivalence between the category of semirings and the category of affine schemes relative to $\Mod_\N^+$ in the sense of To\"en and Vaqui\'e. We consider it together with the Grothendieck topology as defined in section \ref{section: relative schemes}, which we call the \emph{subcanonical Zariski topology}, a name that we justify in the following lemma.

As in the case of subcanonical blue schemes, we say that $\spec C\to \spec B$ is a \emph{finite localization} if $B\to C$ is a finite localization. We say that a family of morphisms $\{U_i\to X\}$ in $\Aff_\N^{+,\can}$ is \emph{subcanonical} if it is a covering family in the canonical topology for $\Aff_\N^{+,\can}$. 

\begin{lemma}\label{lemma: subcanonical coverings of semirings are TV-coverings}
 A family $\{U_i\to X\}$ of finite localizations in $\Aff^{+,\can}_\N$ is a covering family if and only if it is subcanonical.
\end{lemma}

\begin{proof}
 By \cite[Cor.\ 2.11]{Toen-Vaquie09}, every covering family of an affine scheme $\Aff_\N^{+,\can}$ is subcanonical. We proceed with proving the converse statement.
 
 Let $\{U_i\to X\}$ be a covering family of $X=\Spec B$ in $\Aff_\N^{+,\can}$ whose morphisms $U_i\to X$ are dual to localizations $B\to B[h_i^{-1}]$. Let $f:M\to N$ be a morphism in $\Mod^+\N$ such that the induced morphism $f_i:M\otimes_B B[h_i^{-1}] \to N\otimes_B B[h_i^{-1}]$ is an isomorphism for all $i$. We need to show that in this case $f$ is already an isomorphism.
 
 Consider the symmetric algebra $  B[M] = \bigoplus_{i\geq0} \Sym^i(M)$ where
 \[
   \Sym^i(M) \ = \ \underbrace{M\otimes_B \dotsb \otimes_B M}_{i\text{-times}}\,\bigl/\,\bigl\langle{m_1\otimes\dotsb\otimes m_i\sim m_{\sigma(1)}\otimes\dotsb\otimes m_{\sigma(i)} \,\bigl|\, \sigma\in S_i}\bigr\rangle,
 \]
 which is a semiring with respect to the obvious addition and multiplication coming with a natural inclusion $B=M^{\otimes 0}\to B[M]$. The morphism $f:M\to N$ induces a morphism $B[f]:B[M]\to B[N]$ of semirings. If we can show that $B[f]$ is an isomorphism of semirings, then we can conclude the restriction to $M^{\otimes 1}\to N^{\otimes 1}$, which is $f:M\to N$ itself, is an isomorphism in $\Mod^+\N$.
 
 Let $X_M=\spec B[M]$ and $X_N=\spec B[N]$. Since $B[M_i]=B[M]\otimes_B B[h_i^{-1}]$, the fibre product $U_{i,M}=X_M\otimes_X U_i$ is the dual of $B[M_i]$. We obtain for every $i$ a commutative square
 \[
  \xymatrix@R=1pc@C=4pc{U_{i,N} \ar[r]^{B[f_{i}]^\ast} \ar[d] & U_{i,M} \ar[d] \\ X_N \ar[r]^{B[f]^\ast} & X_M.}
 \]
 Since $\{U_i\to X\}$ is a covering family, the pullbacks $\{U_{i,N}\to X_N\}$ and $\{U_{i,M}\to X_M\}$ are covering families as well. Since the $B[f_i]^\ast$ are isomorphisms, we conclude that $B[f]^\ast$ is an isomorphism. This concludes the proof of the lemma.
\end{proof}

In order to compare subcanonical semiring schemes with geometric semiring schemes, we recall a result of Marty that yields yet another characterization of the subcanonical Zariski topology on $\Aff_\N^{+,\can}$.

A subset $I$ of a semiring $B$ is called an \emph{ideal} if it is an $B$-submodule of $B$, i.e.\ if $0\in I$, if $IB=I$ and if $a+b\in I$ whenever $a,b\in I$. Note that every $k$-ideal is an ideal, but the converse is not true in general. A \emph{prime ideal of $B$} is an ideal $\fp$ such that its complement $B-\fp$ is a multiplicative subset.

In analogy to the (geometric) spectrum $\Spec B$ of $B$, which is based on prime $k$-ideals, we define the \emph{subcanonical spectrum} $\Spec^\can B$ as the set of prime ideals of $B$, endowed with the topology generated by the \emph{principal opens}
\[
 U_h^\can \ = \ \{\,\fp\in\Spec^\can B \,|\, h\notin\fp\, \}
\]
for $h\in B$. Note that a morphism $f:B\to C$ of semirings yields a continuous map $f^\ast:\Spec^\can C\to\Spec^\can B$ by taking inverse images of prime ideals.

Marty's result \cite[Thm.\ 3.13]{Marty07} applies to a semiring $B$, seen as a commutative monoid in $\Mod^+\N$, which yields the following.

\begin{thm}\label{thm: locale of a the subcanonical spectrum of a semiring}
 The locale of $\Spec^\can B$ is naturally isomorphic to the locale of $\spec B$. More explicitly, a family $\{B\to B[h_i^{-1}]\}$ of localizations of $B$ defines a covering family $\{\spec B[h_i^{-1}]\to \spec B\}$ if and only if $\Spec^\can B$ is covered by the principal opens $U_{h_i}^\can$ as a topological space.
\end{thm}

As a consequence, we see that the subcanonical Zariski topology on $\Aff_\N^{+,\can}$ is generated by families $\{\spec B[h_i^{-1}]\to \spec B\}$ of finite localizations for which $\Spec^\can B=\bigcup U_{h_i}^\can$.

\begin{rem}
 Marty's result makes it possible to describe subcanonical semiring schemes as blueprinted spaces, or, in this case, as ``semiringed spaces''. We explain this in brevity, but omit proofs since we do not rely on this description in the rest of this paper.
 
 The subcanonical spectrum $X=\Spec^\can B$ comes equipped with a structure sheaf $\cO_X$ that is characterized by $\cO_X(U_{h}^\can)=B[h^{-1}]$. In particular, we have $\cO_X(X)=B$. This endows $\Spec^\can B$ with the structure of a blueprinted space.
 
 A semiring homomorphism $f:B\to C$ and its associated continuous map $f^\ast:Y=\Spec^\can C\to\Spec^\can B=X$ yield a morphism $f^\#:f^{-1}\cO_X\to\cO_Y$ of sheaves in the usual way. It can be shown that this yields a fully faithful embedding $\Aff_\N^{+,\can}\to\bpspaces$. Using similar arguments as in the proof of Lemma \ref{lemma: extension of a morphism of sites to a functor between schemes}, we see that this embedding extends to a fully faithful embedding $\Sch_\N^{+,\can}\to\bpspaces$ whose essential image consists of all blueprinted spaces that are colimits of affine presentations in $\Aff_\N^{+,\can}$. Moreover, the embeddings $\Aff_\N^{+,\can}\to\Sch_\N^{+,\can}\to\bpspaces$ satisfy Hypothesis \ref{hypothesis}. 
\end{rem}

The interpretation of the subcanonical Zariski topology on $\Aff_\N^{+,\can}$ in terms of coverings of the subcanonical spectrum of a semiring yields the promised connection to geometric semiring schemes. By abuse of notation, we use the same symbols $\Gamma^\can:\Aff_\N^{+,\can}\to \bp$ and $\Gamma^\geo:\Aff_\N^{+,\geo}\to \Gamma\bp$ as for blue schemes to denote the global section functors for the respective notions of semiring schemes.

\begin{thm}\label{thm: morphism of sites of affine semiring schemes}
 The functor $\Spec\circ\Gamma^\can:\Aff_\N^{+,\can}\to \Aff_\N^{+,\geo}$ sends covering families to covering families, and $\spec\circ\Gamma^\geo:\Aff_\N^{+,\geo}\to \Aff_\N^{+,\can}$ is a left adjoint and right-inverse to $\Spec\circ\Gamma^\can$.
\end{thm}

\begin{proof}
 We begin with the proof of the latter claim of the theorem. Since both $\Gamma^\can$ and $\Gamma^\geo$ are anti-equivalences, we can derive the assertion from the corresponding property of the globalization functor $\Gamma:\bp\to\Gamma\bp$ and the inclusion $\iota:\Gamma\bp\to\bp$ as full subcategory. It is clear that $\Gamma\circ\iota$ is isomorphic to the identity functor of $\Gamma\bp$. The latter claim of Theorem \ref{thm: globalization induces an isomorphism between spectra} implies that $\iota$ is left adjoint to $\Gamma$.
 
 We proceed with the proof that $\Spec\circ\Gamma^\can:\Aff_\N^{+,\can}\to \Aff_\N^{+,\geo}$ sends covering families to covering families. This can be verified on generators for the Grothendieck pretopology on $\Aff_\N^{+,\can}$, which are of the form $\{\spec B[h_i^{-1}]\to\spec B\}$ where $B$ is a semiring and the morphisms in this family are finite localizations. Since every prime $k$-ideal is a prime ideal, $\Spec B$ occurs as a natural subspace of $\Spec^\can B$. Since the intersection of $U_{h_i}^\can$ with $\Spec B$ is the principal open $U_{h_i}$ of $\Spec B$, the functor $\Spec\circ\Gamma^\can$ maps the covering family $\{\spec B[h_i^{-1}]\to\spec B\}$ to the family of $\{U_{h_i}\to \Spec B\}$ of open embeddings.
 
 By Theorem \ref{thm: locale of a the subcanonical spectrum of a semiring}, the subcanonical spectrum $\Spec^\can B$ is covered by the principal opens $U_{h_i}^\can$ as a topological space. Therefore $\Spec B$ is covered by the open subsets $U_{h_i}$, and $\{U_{h_i}\to \Spec B\}$ is a covering family in $\Aff_\N^{+,\geo}$. This completes the proof of the theorem.
\end{proof}

Theorem \ref{thm: morphism of sites of affine semiring schemes} allows us to apply Lemma \ref{lemma: extension of a morphism of sites to a functor between schemes}, which yields a functor $\cG^+:\Sch_\N^{+,\can}\to \Sch_\N^{+,\geo}$ that sends the colimit $\colim\cU$ of an affine presentation in $\Aff_\N^{+,\can}$ to the colimit $\colim\cU$ in $\Aff_\N^{+,\geo}$. This proves part \eqref{A3} of Theorem A.

Note that since $\Spec\circ\Gamma^\can:\Aff_\N^{+,\can}\to \Aff_\N^{+,\geo}$ is essentially surjective, the functor $\Sch_\N^{+,\can}\to \Sch_\N^{+,\geo}$ is also essentially surjective. We summarize our findings.

\begin{cor}\label{cor: functor from subcanonical semiring schemes to geometric semiring schemes}
 The functor $\Spec\circ\Gamma^\can:\Aff_\N^{+,\can}\to \Aff_\N^{+,\geo}$ extends to an essentially surjective functor $\cG^+:\Sch_\N^{+,\can}\to \Sch_\N^{+,\geo}$. \qed
\end{cor}

\begin{ex}\label{ex: semiring that is not global}
 In the following, we show that the semiring $B=\bpgenquot{\N[a,b,g,h]^+}{g+1\=h,ah\=bg}$ is not a global blueprint.
 
 The spectrum of $B$ has two maximal points, which are the maximal $k$-ideals $\gen{a,b,h}$ and $\gen{a,b,g}$. The respective complements in $\Spec B$ are the open principal open subsets $U_g=\Spec B[g^{-1}]$ and $U_h=\Spec B[h^{-1}]$, whose union covers $\Spec B$. Therefore, we can define the section $s\in \Gamma B$ as $a/g$ on $U_g$ and as $b/h$ on $U_h$. Since $ah=bg$, we have $a/g=b/h$ on the intersection $U_g\cap U_h$, which show that $s$ is indeed a global section of $\Spec B$. 
 
 To see that $s$ does not come from an element in $B$, we multiply the relation $g+1\=h$ by $s$ and use that $sg=a$ and $sh=b$, which yields $s+a\=b$. However, there is no such element $s$ in $B$. This shows that $B\to \Gamma B$ is not an isomorphism.
\end{ex}
 
\begin{rem}
 The previous example of $B=\bpgenquot{\N[a,b,g,h]^+}{g+1\=h,ah\=bg}$ has the following implications. Since $\spec:\bp\to\Aff_\N^{+,\can}$ is an anti-equivalence of categories, $\spec \Gamma B\to \spec B$ is not an isomorphism in $\Aff_\N^{+,\can}$. In contrast, Theorem \ref{thm: globalization induces an isomorphism between spectra} implies that $\Spec \Gamma B\to \Spec B$ is an isomorphism in $\Aff_\N^{+,\geo}$. Using Lemma \ref{lemma: morphisms from global blueprints to other blueprints}, this shows that the functor $\Spec\circ\Gamma^\can:\Aff_\N^{+,\can}\to \Aff_\N^{+,\geo}$ is not full and that $\spec\circ\Gamma^\geo:\Aff_\N^{+,\geo}\to \Aff_\N^{+,\can}$ is not its left-inverse.
 
 Moreover, note that $B$ has a unique maximal ideal $\fm=B-\{1\}$, which is not a $k$-ideal. Thus any covering of $\Spec^\can B$ by principal opens must contain $\Spec^\can B$ itself. In particular, $U^\can_g$ and $U^\can_h$ do not cover $\Spec^\can B$, which shows that the functor $\spec\circ\Gamma^\geo:\Aff_\N^{+,\geo}\to \Aff_\N^{+,\can}$ is not sending covering families to covering families.
\end{rem}

Finally we remark that these effects are particular to semirings and do not occur for rings. To wit, every ideal of a ring is a $k$-ideal. Consequently, the functors $\Spec\circ\Gamma^\can$ and $\spec\circ\Gamma^\geo$ restrict to the well-known mutual inverse equivalences between the category $\Aff_\Z^{+,\can}$ of representable presheaves on $\Rings$ and the category $\Aff_\Z^{+,\geo}$ of affine schemes. Since covering families in $\Aff_\Z^{+,\can}$ coincide with covering families in $\Aff_\Z^{+,\geo}$ under these equivalences, Lemma \ref{lemma: extension of a morphism of sites to a functor between schemes} yields mutual inverse equivalences between the corresponding categories $\Sch_\Z^{+,\can}$ and $\Sch_\Z^{+,\geo}$ of schemes, which we will identify and simply denote by $\Sch_\Z^+$ henceforth.

%%%%%%%%%%%%%%%%%%%%%%%%%%%%%%%%%%%%%%%%%%%%%%%%%%%%%%%%%%%%%%%%%%%%%%%%%%%%%%%%%%%%%%%%%%%%%%%%%%%%%%%%%%%%%%%%%%%%%%%%%%%%%%%%%%%%%%%%%%%%%%%%%%

\section{Compatibility with base extensions}
\label{section: compatibility with base extension}

After we have completed the constructions of the functors $\cG$ and $\cG^+$, we verify in this section the properties that are claimed in \eqref{A1} and \eqref{A2} of Theorem A. We repeat these assertions.

\begin{thm}\label{thm: compatibility with base extensions}
 The diagram of functors 
 \[
  \xymatrix@R=2pc@C=4pc{\Sch_\Fun^\can \ar[rr]^{\cG} \ar[d]^{(-)^+} &   & \Sch_\Fun^\geo \ar[d]<0,6ex>^{(-)^+} \\
                        \Sch_\N^{+,\can} \ar[rr]^{\cG^+}\ar[dr]_{-\otimes_\N\Z} &   & \Sch_\N^{+,\geo} \ar[u]<0,6ex>^{\iota}\ar[dl]^{-\otimes_\N\Z} \\
                                                             & \Sch_\Z^+\ar[ul]<-1ex>_{\iota}\ar[ur]<1ex>^{\iota}  }
 \]
 satisfies the following properties:
  \begin{enumerate}
  \item\label{B1} the outer square and both the inner and the outer triangle commute; 
  \item\label{B2} the embeddings $\iota$ are left inverse and right adjoint to the respective base extension functors $(-)^+$ and $-\otimes_\N\Z$; 
locally of finite type over a blue field.
 \end{enumerate}
\end{thm}

\begin{proof}
 We begin with the proof of \eqref{B2}. The embeddings $\iota:\Sch_\N^{+,\geo}\to\Sch_\Fun^\geo$ and $\iota:\Sch_\Z^+\to\Sch_\N^{+,\geo}$ appear in section 3.6 of \cite{blueprints1}. By \cite[Prop.\ 3.31]{blueprints1}, the former embedding is right adjoint to the base extension $(-)^+:\Sch_\Fun^\geo\to\Sch_\N^{+,\geo}$. An analogous argument shows that the latter embedding is right adjoint to the base extension $-\otimes_\N\Z:\Sch_\N^{+,\geo}\to \Sch_\Z^+$. That in either case the embedding $\iota$ is left inverse to the base extension can be deduced from the corresponding facts for affine schemes (\cite[section 1.4]{blueprints1}) using affine presentations.

 Concerning the subcanonical side of the diagram, Proposition 3.4 of \cite{Toen-Vaquie09} shows that $-\otimes_\N\Z:\Sch_\N^{+,\can}\to\Sch_\Z^+$ has a left adjoint, which is the embedding $\iota:\Sch_\Z^+\to\Sch_\N^{+,\can}$. That $\iota$ is left inverse to $-\otimes_\N\Z$ follows from the corresponding fact for $\iota:\Rings\to\SRings$ and $-\otimes_\N\Z:\SRings\to\Rings$, using affine presentations. 

 This proves all properties asserted in part \eqref{B2} of of the theorem. We continue to verify the claims of part \eqref{B1}.

 To begin with, we note that we can employ affine presentations to reduce the question about the commutativity of subdiagrams of the above diagram to the corresponding question between affine objects. Passing to the dual categories yields the diagram
 \[
  \xymatrix@R=2pc@C=4pc{\bp \ar[rr]^{\Gamma} \ar[d]^{(-)^+} &   & \Gamma\bp \ar[d]<0,6ex>^{\Gamma^\geo\circ(-)^+\circ\Spec} \\
                        \SRings \ar[rr]^{\Gamma}\ar[dr]_{-\otimes_\N\Z} &   & \Gamma\SRings \ar[u]<0,6ex>^{\iota}\ar[dl]^{-\otimes_\N\Z} \\
                                                             & \Rings\ar[ul]<-1ex>_{\iota}\ar[ur]<1ex>^{\iota}  }
 \]
 where $\Gamma\SRings$ is the full subcategory of $\SRings$ whose objects are global semirings. Note that $\Gamma^\can(\spec B)^+=B^+$, but that $\Gamma^\geo(\Spec B)^+=\Gamma(B^+)$ is in general not isomorphic to $B^+$ since it might fail to be global, even if $B$ is a global blueprint; see Example \ref{ex: a global blueprint whose associated semiring is not global} for evidence.

 The outer square in the upper part of the diagram commutes because 
 \[
  \Spec (\Gamma B)^+ \ \simeq \ (\Spec \Gamma B)^+ \ \simeq \ (\Spec B)^+ \ \simeq \ \Spec B^+ \ \simeq \ \Spec \Gamma(B^+),
 \]
 by the definition of $(\Spec B)^+$ as $\Spec B^+$ and by Theorem \ref{thm: globalization induces an isomorphism between spectra}. The outer triangle in the lower part of the diagram commutes for the same reason. The inner triangle in the lower part of the diagram commutes since rings are global blueprints. This verifies part \eqref{B1} of the theorem.
\end{proof}

\begin{ex}\label{ex: a global blueprint whose associated semiring is not global}
 The blueprint $B=\bpgenquot{\Fun[a,b,c,d,g,h]}{ah=bg,g=c+c+h,cd+cd=1}$ is a global blueprint whose associated semiring $B^+$ fails to be global. Indeed, $B$ is global since it is local with maximal $k$-ideal $\gen{a,b,g,h}$. However, $B^+=\bpgenquot{\N[a,b,c,d,g,h]}{ah=bg,g=c+c+h,cd+cd=1}$ is not a global blueprint as can be seen as follows. Since $d$ is the multiplicative inverse of $f=c+c$ in $B^+$, the equation $g=f+h$ implies that $\Spec B^+$ is covered by $U_g$ and $U_h$. Therefore we can define the global section $s$ of $\Spec B^+$ as $a/g$ on $U_g$ and as $b/h$ on $U_h$, which does not come from an element of $B^+$ for similar reasons as explained in Example \ref{ex: semiring that is not global}. This shows that $B^+$ is not global.
\end{ex}

\section{Algebraically presented blue schemes}
\label{section: algebraically presented blue schemes}

When we want to associate a subcanonical blue scheme with a geometric blue scheme, we face two difficulties. Firstly, a finite localization $B\to C$ of blueprints does in general not yield a finite localization $\Gamma B\to \Gamma C$ between their associated global blueprints, cf.\ Example \ref{ex: finite localizations do not commute with globalization}. Secondly, not every covering family $\{U_i\to X\}$ in $\Aff_\Fun^\geo$ is subcanonical. 

In this section, we will introduce the class of algebraically presented blue schemes, that allows us to bridge the gap between geometric blue schemes and subcanonical blue schemes by using an affine presentation that is sufficiently fine.

\begin{ex}\label{ex: finite localizations do not commute with globalization}
 The following is an example of a finite localization $B\to C$ such that the associated map $\Gamma B\to \Gamma C$ between the respective global blueprints is not a finite localization.
 
 Let $B=\bpgenquot{\Fun[a,b,g,h,t]}{g+h\=t,ah\=bg}$. Since $B$ has a unique maximal $k$-ideal $\fm=\langle a,b,g,h,t\rangle$, $\Spec B$ does not have any non-trivial covering. Therefore $B=\Gamma B$ is local. The localization $B[t^{-1}]=\bpgenquot{\Fun[a,b,g,h,t^{\pm1}]}{g+h\=t,ah\=bg}$ has the two maximal $k$-ideals $\fm_g=\langle a,b,h\rangle$ and $\fm_h=\langle a,b,g\rangle$. The respective complements $U_g$ and $U_h$ in $X=\Spec B[t^{-1}]$ cover $X$.
 
 We define the section $s$ in $\Gamma \bigl(B[t^{-1}]\bigr)$ by $s=a/g$ in $U_g=\Spec B[g^{-1},t^{-1}]$ and by $s=b/h$ in $U_h=\Spec B[h^{-1},t^{-1}]$. Since $ah=bg$, this is indeed a well-defined element of $\Gamma \bigl(B[t^{-1}]\bigr)$. We have $s=(g+h)s=a+b$ in $\Gamma \bigl(B[t^{-1}]\bigr)$, but $a+b$ is not an element of $B[t^{-1}]$. In fact, $a+b$ is not contained in any localization of $B$. Therefore the map $\Gamma B=B\to B[t^{-1}]\subset\Gamma \bigl(B[t^{-1}]\bigr)$ is not a finite localization.
\end{ex}

\begin{df}
 A morphism $X\to Y$ of affine geometric blue schemes is a \emph{finite localization} if the morphism $\Gamma Y\to \Gamma X$ of blueprints is a finite localization. An affine geometric blue scheme $X$ is \emph{with an algebraic basis} if it has a unique closed point and if the affine open subsets $U$ of $X$ that are finite localizations form a basis of the topology of $X$. %A blue scheme is \emph{with a localization basis} if every affine open subscheme is with an algebraic basis. 

 A geometric blue scheme $X$ is \emph{algebraically presented} if it has an affine presentation $\cU$ such that all $U$ in $\cU$ are with an algebraic basis and if all morphisms of $\cU$ are finite localizations. We call such an affine presentation $\cU$ an \emph{algebraic presentation of $X$}, and say for short that $X$ is an \emph{algebraically presented blue scheme}, suppressing the attribute ``geometric''. We denote the full subcategory of $\Sch_\Fun$ whose objects are algebraically presented blue schemes by $\Sch_\Fun^\alg$.
\end{df}

\begin{ex}\label{ex: monoidal schemes are algebraically presented}
 A first class of examples are \emph{monoidal schemes}, which are geometric blue schemes that can be covered by the spectra of monoids. Since every monoid $A$ has a unique maximal $k$-ideal, every covering of $\Spec A$ is subcanonical and $A=\Gamma A$. Since every localization of $A$ is also a monoid, we have $\Gamma \bigl(A[h^{-1}]\bigr)=A[h^{-1}]$, which shows that $A$ is with an algebraic basis. We conclude that any affine presentation $\cU$ of a monoidal scheme $X$ such that the objects of $\cU$ are spectra of monoids is an algebraic presentation of $X$.
\end{ex}

Under a certain finiteness assumption, we can broaden the previous example to larger class of geometric blue schemes that include everything that could be considered a variety over $\Fun$; in particular, this includes projective spaces, Grassmannians, toric varieties and algebraic groups.

Let $k$ be a blue field. A \emph{finitely generated blue $k$-algebra} is a morphism $k\to B$ of blueprints such that the underlying monoid of $B$ is finitely generated over $k$ as a semigroup. A \emph{blue $k$-scheme} is a morphism $X\to\Spec k$ of blue schemes, and it is \emph{locally of finite type over $k$} if for every affine open subscheme $U$ of $X$, the induced morphism $k\to \Gamma U$ is a blue $k$-algebra of finite type. If the morphism $X\to \Spec k$ is understood, we suppress it from the notation.

\begin{prop}\label{prop: blue schemes lof over a blue field are algebraically presented}
 If $X$ is a geometric blue scheme that is locally of finite type over a blue field $k$, then it is algebraically presented.
\end{prop}

\begin{proof}
 We begin with the case of an affine blue scheme $X=\Spec B$ where $B$ is a $k$-algebra of finite type. Since every $k$-ideal of $B$ is in particular a $k$-ideal of the underlying monoid $A$ of $B$, $X$ is a subset of $\Spec A$. 
 
 Every prime $k$-ideal $\fp$ of $A$ is of the following form. Fix a set $S$ of generators over $k$. Then $\fp$ is generated as a $k$-ideal of $A$ by a subset of $S$. We conclude that $\Spec A$, and therefore also $\Spec B$, has only finitely many points.
 
 Since $B$ is finitely generated over $k$, the stalk $\cO_{X,x}$ at a point $x$ of $X$ is a finite localization of $B$ and $U_x=\Spec \cO_{X,x}$ is an open subset of $X$. Since $x$ is the unique closed point of $U$, we have $\Gamma U=\cO_{X,x}$, which shows that $U\to X$ is a finite localization. This argument applies, in particular, to $X=U_x$ shows that the specialization maps $U_y\to U_x$ are finite localizations for all points $x$ and $y$ of $X$ where $x$ is contained in the closure of $y$. This shows that $U_x$ is with an algebraic basis, and we conclude that the diagram of all open subsets $U_x=\Spec \cO_{X,x}$ of $X$ together with the specialization maps $U_x\to U_y$ form an affine presentation $\cU$ of $X$.

 If $X$ is an arbitrary blue scheme that is locally of finite type over $k$, then the diagram of all $U_x=\Spec \cO_{X,x}$ together with all specialization maps forms an affine presentation, as can be verified by restricting to affine opens of $X$ and their intersections.
\end{proof}

\begin{rem}\label{rem: algebraically presented blue schemes}
 Note that a blue scheme $X$ is not algebraically presented if it contains a stalk $\cO_{X,x}$ whose spectrum is not an open subset. This applies, for instance, to all varieties in the usual sense over an algebraically closed field that are not a disjoint union of points.

 We observe further that an affine blue scheme with an algebraic basis is algebraically presented, but not vice versa. An example of an affine blue scheme that is algebraically presented, but not with an algebraic basis is $X=\Spec B$ with $B=\bpgenquot{\Fun[a,b,g,h]}{ah\=bg,g+h\=1}$, which is not global since $\Gamma B$ contains a global section $s$ that equals $a/g$ on $U_{g}=\Spec B[g^{-1}]$ and $b/h$ on $U_{h}=\Spec B[h^{-1}]$, cf.\ \cite[Ex.\ 3.8]{blueprints1}. It is not with an algebraic basis because the cover $\{U_{g},U_{h}\}$ is not subcanonical. But the affine presentation that consists of $U_{g}$, $U_{h}$, $U_{gh}$ and the open immersions $U_{gh}\to U_{g}$ and $U_{gh}\to U_{h}$ is an algebraic presentation of $X$. An example of a blue scheme that is not algebraically presented is $X=\Spec B$ with $B=\bpgenquot{\Fun[a_i,b_i,g_i,h_i]_{i\in\N}}{a_ih_i\=b_ig_i,g_i+h_i\=1}_{i\in\N}$ since every open subset has a covering that is not subcanonical.
\end{rem}

The following property will be central for the construction of the canonical blue scheme associated with a blue scheme with algebraic presentation, which is the theme of the following section.

\begin{prop}\label{prop: common refinement for algebraic presentations}
 Let $\cU$ and $\cV$ be two algebraic presentations of a blue scheme $X$ and $\cW'$ a common refinement of $\cU$ and $\cV$. Then there exists a refinement $\cW$ of $\cW'$ such that all morphisms of $\cW$, all morphisms $\Phi_W:W\to \Phi(W)$ in $\Phi:\cW\to\cU$ and all morphisms $\Psi_V: W\to \Psi(W)$ in $\Psi:\cW\to\cV$ are finite localizations.
\end{prop}

We will need some preliminary statements before we can turn to the proof of the proposition.

\begin{lemma}\label{lemma: morphisms from global blueprints to other blueprints}
 Let $B$ be a global blueprint and $f:B\to C$ a morphism of blueprints such that $f^\ast:\Spec C\to \Spec B$ is an isomorphism of affine blue schemes. Then $f$ is an isomorphism of blueprints.
\end{lemma}

\begin{proof}
 Since $\Hom(\Spec C,\Spec B)=\Hom(\Gamma B,\Gamma C)$ and $B=\Gamma B$, we obtain a commutative diagram
 \[
  \xymatrix@R=0pc@C=4pc{B \ar[rr]^\sim \ar[dr]_f && \Gamma C \\ & C \ar[ur]_{\sigma_C}}
 \]
 where the isomorphism $B\to \Gamma C$ is induced by $f^\ast$ and $\sigma_C:C\to \Gamma C$ is the globalization map. This shows that $f$ is injective. If we can show that $\sigma_C$ is also injective, then it is clear that $f$ is an isomorphism.

 To show injectivity, we consider $c$ and $c'$ in $C$ with $\sigma_C(c)=\sigma_C(c')$. Then 
 \[
  \tilde s \ = \ \sigma_C(c)\circ (f^\ast)^{-1} \ = \ \sigma_C(c')\circ (f^\ast)^{-1}: \ \Spec B \quad \longrightarrow \quad \coprod C_\fp \ = \ \coprod B_\fp
 \]
 is a global section of $B$ where $\fp$ ranges through all points of $\Spec C=\Spec B$. Since $B$ is global, there is a unique $b\in B$ such that $\tilde s=\sigma_B(b)$. Therefore we have $c=f(b)=c'$.
\end{proof}

\begin{lemma}\label{lemma: criterium for a localization of a blueprint to be global}
 Let $S\subset B$ be a finitely generated multiplicative subset and $f:B\to C$ and $g:C\to S^{-1}B$ blueprint morphisms such that $g\circ f$ equals the canonical morphism $B\to S^{-1}B$. Let $T=f(S)$. If $f^\ast:\Spec C\to \Spec B$ is an open immersion and $S^{-1}B$ is global, then $f$ induces an isomorphism $f_S:S^{-1}B\to T^{-1}C$ of blueprints. In particular, $T^{-1}C$ is a global blueprint.
\end{lemma}

\begin{proof}
 By the universal property of localizations, $f$ and $g$ induce morphisms $f_S:S^{-1}B\to T^{-1}C$ and $g_T: T^{-1}C\to S^{-1}B$, respectively. Let $U=\Spec B$, $U_S=\Spec S^{-1}B$, $V=\Spec C$ and $V_T=\Spec T^{-1}C$. Then we obtain a commutative diagram
 \[
  \xymatrix{ V_T \ar[rr]\ar@{-->}@<0,5ex>[dr]^{f_S^\ast} & & V\ar[dr]^{f^\ast} \\ & U_S\ar@{-->}@<0,5ex>[ul]^{g^\ast_T}\ar@{-->}[ur]_{g^\ast}\ar[rr] & & U }
 \]
 where the solid arrows are open immersions. This means that $U_S$, $V$ and $V_T$ are open subsets of $U$ and that the respective structure sheaves are restrictions of the structure sheaf of $U$. Consequently, the dashed arrows are open immersions as well. In particular $f_S^\ast$ and $g_T^\ast$ must be mutual inverse isomorphisms. Therefore, we can apply Lemma \ref{lemma: morphisms from global blueprints to other blueprints} to $f_S:S^{-1}B\to T^{-1}C$, which says that $f_S$ is an isomorphism.
\end{proof}

\begin{cor}\label{cor: factorization of finite localizations}
 Let $\varphi:V\to U$ be an open immersion and $\psi:W\to V$ a morphism such that $\psi\circ\varphi: W\to U$ is a finite localization. Then $\psi$ is a finite localization.
\end{cor}

\begin{proof}
 The statement follows from applying Lemma \ref{lemma: criterium for a localization of a blueprint to be global} to $B=\Gamma U$, $C=\Gamma V$, $S^{-1}B=\Gamma W$, $f=\Gamma\varphi$ and $g=\Gamma\psi$.
\end{proof}

\begin{proof}[Proof of Proposition \ref{prop: common refinement for algebraic presentations}]
 Let $\Phi':\cW'\to\cU$ and $\Psi':\cW'\to\cV$ be refinements and $\cU$ and $\cV$ algebraic presentations. We construct $\cW$ as follows. Since $\cU$ is an algebraic presentation, we can cover each $W'\in\cW'_{\max}$ with finite localizations $W_i$ of $U=\Phi'(W')$. Since $\cV$ is an algebraic presentation, we can cover each of the $W_i$ with finite localizations $W_{i,j}$ of $V=\Psi(W')$. By Corollary \ref{cor: factorization of finite localizations}, each $W_{i,j}$ is a finite localizations of $W_i$, and thus of $U$. We define $\cW_{\max}$ as the collection of all $W_{i,j}$ (for varying $W'$), which will be the maximal elements of an affine presentation $\cW$. These sets $W_{i,j}$ come with an open immersions $\xi_{W_{i,j}}:W_{i,j}\to W'$, which will be part of a refinement $\Xi:\cW\to\cW'$.

 For $W'_1,W'_2\in\cW'_{\max}$, $W'\in\cW'_{W'_1,W'_2}$ and $W_1,W_2\in\cW_{\max}$ with $\Xi(W_1)=W_1'$ and $\Xi(W_2)=W_2'$, define $W_{1,2}=W_1\times_{W_1'}W'\times_{W_2'}W_2$, which comes together with open immersions $W_{1,2}\to W_1$, $W_{1,2}\to W_2$ and $\Xi_{W_{1,2}}:W_{1,2}\to W'$. Note that the family of all $W_{1,2}$ for $W_1$ and $W_2$ varying through the open subschemes of the covering of $W_1'$ and $W_2'$, respectively, cover $W'$ by the stability of coverings under base change. Let $U=\Phi'(W')$ and $V=\Psi'(W')$. Then we have the following commutative diagram of open immersions.
 \[
  \xymatrix@R=0pc@C=4pc{ & W_1 \ar[rr]^(0.3){\Xi_{W_1}} & & W_1'  \\ &&&& U \\ W_{1,2} \ar[uur]\ar[rr]^(0.7){\Xi_{W_{1,2}}}\ar[ddr] &&  W' \ar[uur]\ar[ddr]\ar[urr]^(0.75){\Phi'_{W'}}\ar[drr]_(0.75){\Psi'_{W'}} \\ &&&& V \\ & W_2\ar[rr]^(0.3){\Xi_{W_2}} & & W_2' }
 \]
 Since $\Phi'_{W_i'}\circ\Xi_{W_i}:W_i\to W_i'\to U_i$ for $U_i=\Phi'(W_i')$ in $\cU$ and $i=1,2$ are open immersions, Corollary \ref{cor: factorization of finite localizations} implies that the finite localizations of $W_i$ form a basis of its topology. Therefore, we can cover $W_{1,2}$ with finite localizations $W_i$ of $W_1$. We can cover each $W_i$ with finite localizations $W_{i,j}$ of $W_2$, which are finite localizations of $W_i$ by Corollary \ref{cor: factorization of finite localizations} and thus of $W_1$. We can cover each $W_{i,j}$ with finite localizations $W_{i,j,k}$ of $U$, and each $W_{i,j,k}$ with finite localizations $W_{i,j,k,l}$ of $V$. By the same argument as before, Corollary \ref{cor: factorization of finite localizations} implies that each $W_{i,j,k,l}$ is a common finite localization of $W_1$, $W_2$, $U$ and $V$. Note that the family of all $W_{i,j,k,l}$ covers $W_{1,2}$.

 We define $\cW$ as the union of $\cW_{\max}$ together with all sets $W_{i,j,k,l}$, together with the finite localizations $W_{i,j,k,l}\to W_1$ and $W_{i,j,k,l}\to W_2$, where $W_1'$ and $W_2'$ vary through $\cW'_{\max}$, $W'$ varies through $\cW'_{W_1',W_2'}$, $W_1$ and $W_2$ vary through $\cW_{\max}$ such that $\Xi(W_1)=W_1'$ and $\Xi(W_2)=W_2'$, and $W_{i,j,k,l}$ vary through all sets in the covering of $W_{1,2}=W_1\times_{W_1'}W'\times_{W_2'}W_2$. 

 By construction, all morphisms of $\cW$ are finite localizations and the ad hoc-defined set $\cW_{\max}$ is indeed the set of maximal elements of $\cW$. The family of morphisms $\Xi_{W}:W\to \Xi(W)$ defines a morphism $\Xi:\cW\to\cW'$, which is a refinement by the construction of $\cW$. If we define $\Phi=\Phi'\circ\Xi:\cW\to\cU$ and $\Psi=\Psi'\circ\Xi:\cW\to\cV$, then all morphisms $\Phi_W:W\to \Phi(W)$ and $\Psi_W:W\to\Psi(W)$ are finite localizations by the construction of $\cW$. This establishes Proposition \ref{prop: common refinement for algebraic presentations}.
\end{proof}

\begin{cor}\label{cor: presentation of morphisms between algebraically presented bleu schemes}
 Let  $\varphi:X\to Y$ be a morphism between two blue schemes with respective algebraic presentations $\cU$ and $\cV$. Then there exists refinements $\cU'\to\cU$ and $\cV'\to\cV$ with the following properties: there is a morphism $\Phi:\cU'\to\cV'$ of affine presentations that induces $\varphi$; all morphisms of $\cU'$ and $\cV'$ are finite localizations; and the refinements $\cU'\to\cU$ and $\cV'\to\cV$ consist of finite localizations.
\end{cor}

\begin{proof}
 Let $\Phi':\widetilde\cU'\to\widetilde\cV'$ a morphism of affine presentations that induces $\varphi$. Choose a common refinement $\widetilde\cV''$ of $\cV$ and $\widetilde\cV'$. By Proposition \ref{prop: common refinement for algebraic presentations}, we find a refinement $\cV'$ of $\widetilde\cV''$ whose morphisms are finite localizations and such that $\cV'\to\cV$ consists of finite localizations. Then $\widetilde\cU''=\cV'\times_{\widetilde\cU'}\widetilde\cV'$ is a refinement of $\widetilde\cU'$ and it comes with a morphism $\Phi'':\widetilde\cU''\to\cV'$ that induces $\varphi$. Let $\cU''$ be a common refinement of $\cU$ and $\widetilde\cU''$. By Proposition \ref{prop: common refinement for algebraic presentations}, there is a refinement $\cU'$ of $\widetilde\cU''$ whose morphisms are all finite localization and such that the refinement $\cU'\to\cU$ consists of finite localizations. It is clear that the composition $\Phi:\cU'\to\cU''\to\widetilde\cU''\to\cV'$ induces $\varphi$, which concludes the 
proof of the corollary.
\end{proof}

%%%%%%%%%%%%%%%%%%%%%%%%%%%%%%%%%%%%%%%%%%%%%%%%%%%%%%%%%%%%%%%%%%%%%%%%%%%%%%%%%%%%%%%%%%%%%%%%%%%%%%%%%%%%%%%%%%%%%%%%%%%%%%%%%%%%%%%%%%%%%%%%%%

\section{Algebraically presented blue schemes as subcanonical blue schemes}
\label{section: algebraically presented and subcanonical blue schemes}

The idea for associating a subcanonical blue scheme with a geometric blue scheme $X$ is to consider the colimit $\colim (\spec\circ\Gamma)(\cU)$ where $\cU$ is an affine presentation for $X$. Though in general this construction depends on the choice of affine presentation, we will see in this section that it gives rise to a well-defined functor
\[
 \cF: \ \Sch_\Fun^\alg \ \longrightarrow \ \Sch_\Fun^\can
\]
from the full subcategory $\Sch_\Fun^\alg$ of algebraically presented blue schemes in $\Sch_\Fun^\geo$, given that we assume that $\cU$ is an algebraic presentation of $X$.

\subsection*{Objects}
Let $X$ be a blue scheme with an algebraic presentation $\cU$. Then all morphisms of the dual diagram $\Gamma\cU$ in the category of blueprints are finite localizations, and therefore $\spec\Gamma\cU$ is a diagram of Zariski opens in $\Sch_\Fun^\can$ and $\spec\Gamma\cU$ is an affine presentation. We define $\cF(X)=\colim\spec\Gamma\cU$.

This definition is independent (up to canonical isomorphism) from the choice of algebraic presentation for the following reason. Let $\cU$ and $\cV$ be two algebraic presentations of $X$, and $\cW'$ a common refinement of $\cU$ and $\cV$. By Proposition \ref{prop: common refinement for algebraic presentations}, there is a refinement $\cW$ of $\cW'$ such that all morphisms in $\cW$, all morphisms $W\to\Phi(W)$ in $\Phi:\cW\to\cU$ and all morphisms $W\to\Psi(W)$ in $\Psi:\cW\to\cV$ are finite localizations. Therefore, $\spec\Gamma\cW$ is an affine presentation in the category of affine canonical blue schemes. The induced morphisms $\spec\Gamma\Phi: \spec\Gamma\cW\to\spec\Gamma\cU$ and $\spec\Gamma\Psi: \spec\Gamma\cW\to\spec\Gamma\cV$ are refinements since all morphisms $\Phi_W$ and $\Psi_W$ are finite localizations and since all coverings of an object $U\in\cU$ (or $V\in\cV$, respectively) are canonical by the definition of an algebraic presentation.

\subsection*{Morphisms}
Let $\varphi:X\to Y$ be morphism of schemes with algebraic presentations $\cU$ and $\cV$. By Corollary \ref{cor: presentation of morphisms between algebraically presented bleu schemes}, there are refinements $\cU'\to\cU$ and $\cV'\to\cV$ such that all morphisms involved are finite localizations and such that $\varphi$ is induced by a morphism $\Phi:\cU'\to\cV'$ of affine presentations. This means that $\spec\Gamma\cU'$ and $\spec\Gamma\cV'$ are affine presentations in the category of affine canonical blue schemes. Since $\cU$ and $\cV$ are algebraic presentations, $\spec\Gamma\cU'\to\spec\Gamma\cU$ and $\spec\Gamma\cV'\to\spec\Gamma\cV$ are refinements. Thus $\cF(X)$ can be presented as the colimit of $\colim\spec\Gamma\cU'$ and $\cF(Y)$ can be presented as the colimit of $\colim\spec\Gamma\cV'$. We define $\cF(\varphi):\cF(X)\to\cF(Y)$ as the colimit of the induced morphism $\spec\Gamma\Phi:\spec\Gamma\cU'\to \spec\Gamma\cV'$ of affine presentations.

The independence of this definition from the chosen affine presentations $\cU$, $\cV$, $\cU'$ and $\cV'$ can be seen by considering suitable common refinements. We omit the arguments, which are similar to the ones that we used before. Since the definition of $\cF$ is stable under refinements, it follows that $\cF(\varphi\circ\psi)=\cF(\varphi)\circ\cF(\psi)$. 

\begin{rem}
 Note that the image $\cF(X)$ is not affine if $X$ has more than one closed point. This is, in particular, the case for the spectra of blueprints with more than one maximal $k$-ideal and shows that $\cF$ does not restrict to a functor from $\Aff_\Fun^\alg$ to $\Aff_\Fun^\can$.
\end{rem}

Let $\cG:\Sch_\Fun^\can\to\Sch_\Fun^\geo$ denote the extension of the functor $\Spec\circ\Gamma:\Aff_\Fun^\can\to\Aff_\Fun^\geo$, as introduced in section \ref{section: compatibility with base extension}.

\begin{thm}\label{thm: comparison theorem}
 The functor $\cG\circ\cF:\Sch_\Fun^\alg\to\Sch_\Fun^\geo$ is isomorphic to the embedding of $\Sch_\Fun^\alg$ as a subcategory of $\Sch_\Fun^\geo$. The functor $\cF:\Sch_\Fun^\alg\to\Sch_\Fun^\can$ is fully faithful.
\end{thm}

\begin{proof}
 Let $X$ be a blue scheme with algebraic presentation $\cV$. Then $\cF(X)$ is isomorphic to the colimit of the affine presentation $\spec\Gamma\cU$. By the definition of $\cG$, $\cG(\cF(X))$ is isomorphic to the colimit of the affine presentation $\Spec\bigl(\spec\Gamma\cU\bigr)^\op$, which is canonically isomorphic to $\cU$ itself. This means that $\cG\circ\cF(X)$ is canonically isomorphic to $X$.

 Let $\varphi:X\to Y$ be a morphism of blue schemes with algebraic presentations $\cU$ and $\cV$, respectively. Then $\varphi$ is induced by a morphism $\Phi:\cU'\to\cV'$ of refinements of $\cU$ and $\cV$, respectively, which induce isomorphisms $\colim\spec\Gamma\cU'\to\cF(X)$ and $\colim\spec\Gamma\cV'\to\cF(Y)$. This means that $\cF(\varphi)$ is represented by $\spec\Gamma\Phi:\spec\Gamma\cU'\to\spec\Gamma\cV'$. This induces an isomorphism of $\cG(\cF(\varphi))$ with $\Spec\bigl(\spec\Gamma\Phi\bigr)^\op$, which is canonically isomorphic to $\varphi$ itself. This shows the former claim of the theorem.

 If $X$ and $Y$ are in the essential image of $\cF$ and $\varphi:X\to Y$ is a morphism of canonical blue schemes, then it is clear from the local nature of $\cF$ and $\cG$ that $\cF(\cG(\varphi))$ is naturally identified with $\varphi$. This shows that $\cF$ is fully faithful.
\end{proof}

%%%%%%%%%%%%%%%%%%%%%%%%%%%%%%%%%%%%%%%%%%%%%%%%%%%%%%%%%%%%%%%%%%%%%%%%%%%%%%%%%%%%%%%%%%%%%%%%%%%%%%%%%%%%%%%%%%%%%%%%%%%%%%%%%%%%%%%%%%%%%%%%%%

\section{Concluding remarks}
\label{section: concluding remarks}

\subsection{} The functors $\cF$ and $\cG$ restrict to mutual inverse equivalences between monoi\-dal schemes and schemes relative to the category of pointed sets, which is, up to the technical variance of the base point, Vezzani's main result in \cite{Vezzani12}.

\subsection{} We can consider the functor $\cF$ as a \emph{local section} to the functor $\Spec\circ\Gamma:\Sch_\Fun^\can\to\Sch_\Fun^\geo$. There are other choices of local sections to $\Spec\circ \Gamma$, one of which is the following.

Let us call a blueprint $B$ \emph{totally global} if every localization of $B$ is global. Examples of such blueprints are monoids and rings. A geometric blue scheme $X$ is \emph{totally global} if every open affine subscheme is the spectrum of a totally global blueprint. Examples are monoidal schemes and usual schemes. There are, however, blue schemes of finite type over $\Fun$ that are not totally global, e.g.\ $X=\Spec \bpgenquot{\Fun[a,b,g,h]}{g+h\=1,ah\=bg}$. Thus the class of totally global blue schemes does not include all scheme of finite type over $\Fun$.

Let $\Sch_\Fun^\tot$ be the full subcategory of $\Sch_\Fun^\geo$ whose objects are totally global blue schemes. The diagram $\cU$ of all open subschemes of a totally global blue scheme $X$, together with all inclusion maps, forms an affine presentation of $X$ such that $\spec\circ\Gamma\cU$ is an affine presentation in $\Aff_\Fun^\can$. The association $\cF'(X)=\colim\spec\circ\Gamma(\cU)$ is functorial and defines a the local section
\[
 \cF':\ \Sch_\Fun^\tot \quad \longrightarrow \quad \Sch_\Fun^\can
\]
to $\Spec\circ\Gamma:\Sch_\Fun^\can\to\Sch_\Fun^\geo$. It restricts to a functor $\Aff_\Fun^\tot \to\Aff_\Fun^\can$, and it coincides with $\cF$ on the subcategory of monoidal schemes, but not in general. The blueprint $B=\bpgenquot{\Fun[a,b]}{a+b\=1}$ is an example for which $X=\Spec B$ is algebraically presented and totally global, but $\cF'(X)\ncong\cF(X)$.

\bibliographystyle{plain}

\end{document}